\documentclass[12pt, reqno]{amsart}
\usepackage{amsmath,amssymb}
\usepackage{enumitem}
\usepackage{graphicx}
\usepackage{caption}
\usepackage{amsmath,stackengine}

\usepackage{mathrsfs}
\usepackage{mathtools}          
\usepackage{esint}  
\usepackage{amssymb}
\usepackage{xcolor}
\usepackage{dsfont}

\usepackage{enumitem}

\DeclareMathOperator*{\argmin}{arg\,min}

\DeclareMathOperator*{\esssup}{ess\,sup}
\DeclareMathOperator*{\essinf}{ess\,inf}

\usepackage{hyperref}

\DeclareGraphicsExtensions{.pdf,.png,.jpg}

\def\Xint#1{\mathchoice
{\XXint\displaystyle\textstyle{#1}}%
{\XXint\textstyle\scriptstyle{#1}}%
{\XXint\scriptstyle\scriptscriptstyle{#1}}%
{\XXint\scriptscriptstyle\scriptscriptstyle{#1}}%
\!\int}
\def\XXint#1#2#3{{\setbox0=\hbox{$#1{#2#3}{\int}$ }
\vcenter{\hbox{$#2#3$ }}\kern-.6\wd0}}

\def\dashint{\Xint-}

\newtheorem{theorem}{Theorem}[section]
\newtheorem{proposition}[theorem]{Proposition}  

\newtheorem{lemma}[theorem]{Lemma}

\theoremstyle{definition}
\newtheorem{assume}[theorem]{Assumption}
\newtheorem{definition}[theorem]{Definition}

\usepackage{color}

\theoremstyle{example}
\theoremstyle{remark}
\newtheorem{remark}[theorem]{Remark}
\numberwithin{equation}{section}

\newcommand{\R}{{ \mathbb{R}  }}

\newcommand{\N}{{ \mathbb{N}  }}
\newcommand{\p}{{ \partial  }}

\newcommand{\hrho}{{ \hat{\rho}  }}

\newcommand{\Om}{{ \Omega  }}
\newcommand{\OM}{\overline{\Omega}}
\newcommand{\nn}{{ \nonumber }}

\newcommand{\Del}{{ \Delta  }}

\newcommand{\lam}{{ \lambda  }}
\newcommand{\Lam}{{ \Lambda  }}
\newcommand{\cA}{{\mathcal{A}}}

\newcommand{\cE}{{\mathcal{E}}}

\newcommand{\cL}{{\mathcal{L}}}
\newcommand{\cS}{{\mathcal{S}}}

\newcommand{\del}{\delta}
\newcommand{\gam}{\gamma}
\newcommand{\Gam}{\Gamma}
\newcommand{\vphi}{\varphi}
\newcommand{\vrho}{\varrho}
\newcommand{\cM}{\mathcal{M}}
\newcommand{\si}{\sigma}
\newcommand{\grad}{\nabla}

\newcommand{\T}{\tau}
\newcommand{\Tk}{\tau_{k}}
\newcommand{\cK}{\mathcal{K}}

\newcommand{\ra}{\rightarrow}	
\newcommand{\da}{\downarrow}	
\newcommand{\FR}{\textrm{FR}}
\newcommand{\WFR}{\textrm{WFR}}

\newcommand{\half}{ {n+1/2}^\T}
\newcommand{\one}{ {n+1}^\T }

\newcommand{\placek}{ {n}^{\Tk}}
\newcommand{\halfk}{ {n+1/2}^{\Tk}}
\newcommand{\onek}{ {n+1}^{\Tk} }

\newcommand{\ed}{ {\ep,\del} }

\newcommand{\al}{\alpha}
\newcommand{\bt}{\beta}
\newcommand{\ep}{\varepsilon}

\newcommand{\wra}{\rightharpoonup}
\newcommand{\wsra}{\overset{\ast}{\rightharpoonup}}

\def\dashint{\Xint-}

\def\BigRoman{\uppercase\expandafter{\romannumeral\number\count
255 }}
\def\Romannumeral{\afterassignment\BigRoman\count255=}

\newenvironment{thm1.5}{{\par\noindent\bf
           Proof of Theorem \ref{weak-1}. }}
           {\hfill\fbox{}\par\vspace{.2cm}}

\newenvironment{thm1.6}{{\par\noindent\bf
           Proof of Theorem \ref{weak-2}. }}
           {\hfill\fbox{}\par\vspace{.2cm}}

\newenvironment{thm1.7}{{\par\noindent\bf
           Proof of Theorem \ref{weak-4}. }}
           {\hfill\fbox{}\par\vspace{.2cm}}

\newenvironment{thm1.8}{{\par\noindent\bf
           Proof of Theorem \ref{bdd weak-1}. }}
           {\hfill\fbox{}\par\vspace{.2cm}}

\newenvironment{thm1.9}{{\par\noindent\bf
           Proof of Theorem \ref{bdd weak-2}. }}
           {\hfill\fbox{}\par\vspace{.2cm}}

\newenvironment{thm1.10}{{\par\noindent\bf
           Proof of Theorem \ref{bdd weak-4}. }}
           {\hfill\fbox{}\par\vspace{.2cm}}

\usepackage[scale=0.75,hmarginratio=1:1,vmarginratio=2:3,a4paper]{geometry}

\author[Kyungkeun Kang, Hwa Kil Kim, Geuntaek Seo]{Kyungkeun Kang, Hwa Kil Kim, Geuntaek Seo}
\address{Kyungkeun Kang: Department of Mathematics \newline Yonsei University, Seoul 03722, Republic of Korea}
\email{kkang@yonsei.ac.kr}

\address{Hwa Kil Kim: Department of Mathematics Education \newline Hannam University, Daejeon 34430, Republic of Korea}
\email{hwakil@hnu.kr}

\address{Geuntaek Seo:  Department of Mathematics\newline POSTECH, Pohang 37673,  Republic of Korea}
\email{gtseo@postech.ac.kr}

\begin{document}
\captionsetup[figure]{labelfont={rm},labelformat={default},labelsep=period,name={Fig.}}

\title[Keller-Segel models with logistic source]{Bounded weak solutions for Keller-Segel equations with generalized diffusion and logistic source via an unbalanced optimal transport splitting scheme}

\begin{abstract} 
We consider a parabolic-elliptic type of Keller-Segel equations with generalized diffusion and logistic source under homogeneous Neumann-Neumann boundary conditions.
We construct bounded weak solutions globally in time in an unbalanced optimal transport framework, provided that the magnitude of the chemotactic sensitivity can be restricted depending on parameters.
In the case of subquadratic degradation of the logistic source, we quantify the chemotactic sensitivity, in particular, in terms of the power of degradation and the pointwise bound of the initial density.
\end{abstract}

\maketitle
\section{Introduction}
\subsection{Keller-Segel models \& gradient flows approach}
In this paper, we study the following Keller-Segel equations with generalized diffusion and logistic source:
\begin{equation}\label{KSL}
  \begin{cases}
          \p_t \rho= \grad \cdot( \rho \grad U'(\rho) - \chi \rho \grad c) -  \rho F'(\rho)  \\
\hfil      -\Delta c + \Lam c = \rho      
  \end{cases}
\end{equation}
The equation \eqref{KSL} is set in $\Om \times (0,T)$, $T>0$,  where $\Om$ is a bounded domain in $\R^d$. Here, $\rho$ and $c$ denote the density of bacteria and the concentration of a chemoattractant substance, respectively.
The constant $\chi> 0$ represents the sensitivity of the bacteria to the chemoattractant, and $\Lam>0$ denotes the degradation rate of the chemoattractant. The homogeneous Neumann-Neumann boundary conditions are imposed, and the initial datum $\rho_0 \in L^1(\Om)$ is given, i.e.,
\begin{equation}\label{Homo_Neumann}
\rho \grad U'(\rho) \cdot \nu= \grad c \cdot \nu=0, \qquad \rho(x,0)=\rho_0(x).
\end{equation} 
We assume that $U:[0,\infty) \ra \R$ and $F:[0,\infty) \ra \R$ in \eqref{KSL} are convex functions which will be specified later.
It is taken into account that the logistic term $-\rho F'(\rho)$  slows the rate of proliferation when the density of organisms is large.

We note that the typical cases of $U$ and $F'$ are
\begin{equation}\label{typical_UF}
U(\rho)=\rho \log \rho, \qquad  F'(\rho)=  \bt \rho^{r-1} -\al,
\end{equation}
where $\al\geq0,~ \bt>0$, and $r>1$.
In such a case, the Keller-Segel system \eqref{KSL} becomes
\begin{equation}\label{classical_KSL}
  \begin{cases}
          \p_t \rho= \Delta \rho - \chi \grad \cdot(\rho \grad c) +\al \rho - \bt \rho^r  \\
\hfil      -\Delta c + \Lam c = \rho  \\
      \end{cases}
\end{equation}
which has been extensively studied (see e.g. \cite{TW07, W18, WMZ14,  Z15,  SS17, CZ14, W14}).
In recent decades, one of the central issues for these equations (including equations without logistic source terms) has been whether the solutions of \eqref{classical_KSL} are globally bounded or blow up in a finite time. Some notable results are summarized below:

$\bullet$ It was proved in \cite{TW07}  that in the case $r > 2$, a unique classical solution globally exists.  On the other hand, when $r=2$, it was also shown in \cite{TW07} that if $\bt>  \frac{\chi(d-2)}{d}$ (the limiting case $\bt=  \frac{\chi(d-2)}{d}$ was resolved in \cite{KS16}), then the same conclusion holds as in the case $r>2$ (see also \cite{WMZ14, Z15}).

$\bullet$ 
It was shown in \cite{Wink08} that if $2- \frac{1}{d}<r<2$,  very weak solutions exist globally in time 
when the logistic source is in principle of the form $\al-\bt \rho^r$.
Furthermore, it was also proved that a global bounded very weak solution exists if $||\rho_0||_{L^\gamma}$ is sufficiently small for $\gam > \max(1, d/2).$

$\bullet$
It was shown in \cite{W18} that for $\chi = \Lam= 1$, $d \geq3$, and $\Om= B_1(0)$, there exists $\kappa= \kappa(d)<2$ such that
if  $r \in (1, \kappa)$, then
blow-up may occur in a finite time.

$\bullet$
The paper \cite{CS18} establishes an $L^\infty$ estimate for the discrete-in-time solutions of the JKO scheme for the Keller-Segel system ($F=0$), which in turn guarantees the short-time existence of a continuous weak solution (see also \cite{GLM19} for a similar analysis on reaction-advection-diffusion equations using splitting JKO schemes).

\smallskip
The main objective of this paper is to construct global-in-time bounded weak solutions of 
\eqref{KSL}-\eqref{Homo_Neumann} by employing 
gradient flow structures, and to quantify $\chi$ in terms of given parameters and the size of $||\rho_0||_{L^\infty}$ (see Theorem \ref{main}).
To this end, we decompose the governing equation into a transport part (divergence term) and a reaction part (source term). We then analyze the gradient flow structure of each part and employ the JKO splitting scheme to establish a uniform $L^\infty$ estimate for the solutions (see \eqref{CS_estimate} and \eqref{essential_key_estimate}).

To introduce gradient flow structures, we define energy functionals associated with gradient flows, starting with the assumptions on $U$ and $F$.

\begin{assume}\label{assumptionsUF}
The function $U$ satisfies
\begin{enumerate}[label=(\roman*)]
\item[(H1)]$U \in C^2(0,\infty) \cap C[0,\infty)$.

\item[(H2)] $U'' > 0$ on $(0,\infty)$.

\item[(H3)] 
$\lim_{s \ra 0+} sU'(s)$ exists.
\end{enumerate}

The function $F$ satisfies
\begin{enumerate}
\item[(A1)] $F \in C^1[0,\infty) \cap C^2(0,\infty)$, and $F''>0$ on $(0,\infty)$.

\item[(A2)] There exist $\al \geq0,~ \bt>0$, and $r>1$ such that $F'(s) \geq \bt s^{r-1} - \al$.

\item[(A3)] 
If $u \in L^1(\Om)$ and $F(u) \in L^1(\Om)$, then $uF'(au) \in L^1(\Om)$ for any $a>0$. 

\end{enumerate}

\end{assume}

\begin{remark}
It is worth noting that the superlinear growth of $U$ is not necessary at infinity, since all functions in the admissible set are uniformly bounded in JKO scheme (see \eqref{argmin}).

\end{remark}

We denote by $\cM^+(\Om)$ the set of nonnegative Radon measures on $\Omega$.
Suppose that  $U$ and  $F$ fulfill Assumption \ref{assumptionsUF}. 
Given $\rho \ll \cL^d|_\Om$, let $c=c[\rho]$ be the solution of 
\begin{equation}\label{elliptic}
  \begin{cases}
-\Delta c + \Lam c = \rho  \quad   & \text{in $\Om$}, \\
\hfil   \frac{\p c}{\p \nu}=0 & \text{on $\p \Om$.} 
  \end{cases}
\end{equation}

\noindent
We denote
\begin{equation}\label{K_U}
\mathscr{K}_U:=\Big\{ \rho \in \cM^+(\Om): \rho \ll \cL^d|_{\Om} \quad \textrm{and} \quad \int_\Om U(\rho) dx <\infty \Big\}
\end{equation}
and define the energy functional $\cE_1$ on $\cM^+(\Om)$ as follows:
\begin{equation}\label{energy}
  \cE_{1}(\rho):=
  \begin{cases}
                                   \displaystyle{\int_{\Om}} U \big( \rho(x) \big)  dx - \frac{\chi}{2} \int_\Om
                                    \Lam \big|c(x) \big|^2  + |\grad c(x)|^2 dx& \text{if $\rho \in  \mathscr{K}_U \cap L^\infty(\Om)$, } \\
\hfil                                   \infty & \text{else.} 
  \end{cases}
\end{equation}

With the set $\mathscr{K}_F$, similarly defined in \eqref{K_U}, we define
\begin{equation}\label{energy2}
  \cE_2(\rho):=
  \begin{cases}
                                   \displaystyle{\int_{\Om}} F(\rho(x)) dx  & \text{if $\rho \in \mathscr{K}_F$}, \\
                              \hfil     \infty & \text{else.} 
  \end{cases}
\end{equation}

\begin{remark}
In addition to \eqref{typical_UF},
let us comment on the typical types of $U$ and $F$.
For example, if we set $U(s)= \frac{1}{m-1}s^m$, $m>0$, $m\neq1$, then the term $\grad \cdot (\rho \grad U'(\rho))$ in \eqref{KSL} becomes the porous media diffusion, i.e., $\Delta \rho^m$. Furthermore, with a suitable $U$, $\grad \cdot (\rho \grad U'(\rho))$ can represent a linear diffusion (e.g. $U(s)=s \log s$).
From (A2), we note that the logistic source term $-\rho F'(\rho)$ does not exceed $\al \rho - \bt \rho^r$ (if  $F(s):= \bt s^r /r - \al s$, then $-\rho F'(\rho)$ coincides with the logistic term in \eqref{classical_KSL}, i.e., $\al \rho -\bt \rho^r$).
\end{remark}

Let us mention our approaches. 
We analyze the solution of \eqref{KSL} by splitting the equation and by finding two gradient flow structures in the equation.
For convenience, we write $W_2$ and $\FR$ as the 2-Wasserstein metric and Fisher-Rao metric, respectively (the precise definition of these metrics will be specified in Section \ref{section_pre}).
In short, our strategy is to regard \eqref{KSL} formally as
$$
\p_t \rho = -\textrm{grad}_{W_2} \cE_1(\rho) -\textrm{grad}_{\FR} \cE_2(\rho).
$$
where $\cE_1$ and $\cE_2$ are defined in \eqref{energy} and \eqref{energy2} respectively.
The main point is to use alternately gradient flow structures of $\cE_1$ and $\cE_2$ to construct a sequence of approximated solutions in $(\cM^+(\Om), \WFR)$, where $\WFR$ denotes the \textit{Wasserstein-Fisher-Rao} metric (also called \textit{Hellinger-Kantorovich} distance,
see Definition \ref{WFR_distance}). This will lead to the global existence of bounded weak solutions under the suitable restriction on $\chi$ (see \eqref{chi_star}).

To be more precise, firstly,  
$W_2$ gradient flows of $\cE_{1}$, i.e., $\p_t  \rho = -\textrm{grad}_{W_2} \cE_{1}(\rho)$, read (see e.g. \cite{Blanchet13})
\begin{equation}\label{GF1}
\p_t \rho = \grad \cdot( \rho \grad U'(\rho)) - \chi \grad \cdot(\rho \grad c[\rho]).
\end{equation}
On the other hand, similarly, $\FR$ gradient flows of $\cE_2$, i.e., $\p_t  \rho = -\textrm{grad}_{\FR}\cE_2(\rho)$ read (see e.g. \cite{GM17, GLM19})
\begin{equation}\label{GF2}
\p_t \rho= -\rho F'(\rho).
\end{equation}

The results of \cite{CS18} give $L^\infty$ estimates for densities locally in time by using the JKO scheme (see \cite{JKO98}) to exploit the gradient flow structure of \eqref{GF1}. 
In addition, the authors \cite{GLM19} obtained the existence of weak solutions for the types of reaction-advection-diffusion equations
$$
\p_t \rho = \grad \cdot \big( \rho \grad(F_1'(\rho) + V_1) \big) - \rho \big(F_2'(\rho) + V_2 \big)
$$
with some specific $F_1$ and $F_2$ and with $V_1 \in W^{1, \infty}(\Om)$ and $V_2 \in L^\infty(\Om)$,
by means of splitting JKO schemes for the Wasserstein metric and for the Fisher-Rao metric (compare to \cite{GM17}).

Motivated by the approach in \cite{GLM19}, we construct minimizers of Wasserstein/Fisher-Rao JKO schemes by splitting \eqref{KSL} into two parts given in \eqref{GF1} and \eqref{GF2}.

The main difference compared to \cite{GLM19} is that the chemotactic advection term in \eqref{KSL}, $-\grad \cdot(\chi \rho \grad c[\rho])$,  is of nonlinear structure, which causes major obstacles and makes our problem more difficult.

In our case, since our main interest is to establish global existence of bounded weak solutions, it seems necessary to obtain uniform control for $L^\infty$ norms of minimizers obtained by each splitting scheme, which is, however, not necessary in \cite{GLM19} because of the linear structure of an advection term, $\nabla \cdot (\rho \nabla V_1)$, and $V_1\in W^{1,\infty}$. 
Therefore, the novelty of our paper lies in the construction of weak solutions with uniform bounds via the Wasserstein-Fisher-Rao metric under the size restriction of $\chi$.

To be more precise,  we use the result of \cite{CS18} to obtain the $L^\infty$ estimate of the minimizers given by the Wasserstein JKO scheme corresponding to \eqref{GF1}, and then control $L^\infty$ norm of the minimizers caused by the Fisher-Rao JKO step relevant to \eqref{GF2}.
We note that $L^\infty$ norm of minimizers obtained from the Wasserstein JKO scheme in \cite{CS18} generally explodes in a finite time.
With the aid of the damping effect of the logistic source, the size of the mass and the $L^\infty$ norm of minimizers can be uniformly controlled by going through the Fisher-Rao JKO scheme, provided that $\chi$ is less than a threshold depending on other parameters.
As a consequence of the uniform boundedness, we obtain the existence of bounded weak solutions of \eqref{KSL}.

\subsection{Main results}
Throughout this paper, we write
\begin{equation}\label{chi_star}
\chi_*
:=
\begin{cases}
   \infty & \text{if $r>2$,} \\
    \bt & \text{if $r=2$,}\\
    ||\rho_0||_{L^\infty}^{-1} (\bt ||\rho_0||_{L^\infty}^{r-1}  - \al)  & \text{if $1<r<2$ \quad and \quad $||\rho_0||_{L^\infty} > \big( \frac{\al}{\bt(2-r)} \big)^{\frac{1}{r-1}} $,} \\
    
    \al^{\frac{2-r}{1-r}} \bt^{\frac{1}{r-1}}(2-r)^{\frac{2-r}{r-1}} (r-1)  & \text{if $1<r<2$ \quad and \quad $||\rho_0||_{L^\infty} \leq \big( \frac{\al}{\bt(2-r)} \big)^{\frac{1}{r-1}}   $}.
     \end{cases}
    \end{equation}

Next, we introduce the notion of weak solutions of \eqref{KSL}.
\begin{definition}\label{def_weak_solution}
Let $\rho_0 \in L^1(\Om)$. 
A pair $(\rho, c)$ is called a weak solution of \eqref{KSL}-\eqref{Homo_Neumann} if  for all $T>0$
\begin{enumerate}[label=(\roman*)]
\item
$\rho, ~\rho \grad U'(\rho), ~ \rho \grad c[\rho]$, and $\rho F'(\rho)$ belong to $L^1 \big( \Om \times (0,T) \big)$,
where $c[\rho]$ solves \eqref{elliptic} in the sense of distributions.

\item
for all $\vphi \in C_c^\infty( \OM  \times [0,T))$
\begin{equation}\label{real_weak_solution}
\int_0^T \int_\Om  \Big( \rho \p_t \vphi - \rho \grad U'(\rho) \cdot \grad \vphi  +\chi \rho \grad c[\rho] \cdot \grad \vphi -\rho F'(\rho)\vphi \Big) dxdt 
= -\int_\Om \vphi(x,0) \rho_0  dx.
\end{equation}
\noindent

\end{enumerate}

\end{definition}

We are now ready to state our main theorem:
\begin{theorem}\label{main}
Suppose that $\Om$ is a strictly convex bounded domain in $\R^d$ with smooth boundary and assume that $U$ and $F$ are functions satisfying Assumption \ref{assumptionsUF}.
Let $\rho_0 \in L^\infty(\Om) \cap \cM^+(\Om)$. If $0< \chi < \chi_*$, 
there exists a weak solution $(\rho,c[\rho])$ of \eqref{KSL}-\eqref{Homo_Neumann}.
In particular, for all $T>0$, the curve $\rho :[0,T] \ra \cM^+(\Om)$ is $\frac{1}{2}$-H\"{o}lder continuous with respect to \emph{\WFR}-metric, and $\rho$ satisfies
\begin{equation}\label{The_estimate}
\rho \grad U'(\rho) \in L^2 (\Om \times (0,T))
\qquad and \qquad
||\rho||_{L^\infty(\Om \times [0,T])} \leq C,
\end{equation}
where $C>0$ is a constant independent of $T$.
\end{theorem}

\begin{remark}
We make some comments on the existence of blow-up/bounded solutions in \cite{Wink08, W18}, which are most relevant to our results.
The only case $\chi=1$ was treated in \cite{W18}, but it is not difficult to see that the result can be also extended to arbitrary positive $\chi$.
In such a case, the main result in \cite{W18} can be restated as follows: When $\Om= B_R(0)$, $r \in (1,\frac{7}{6})$, and $d =3,4$ (or $r \in (1, 1+\frac{1}{2(d-1)})$ with $d\geq5$), there exist $C>0$ and $R_0>0$ such that if
$\chi \dashint_{B_{R_0}(0)} \rho_0 \geq C$, then a finite-time blow-up occurs, where $\dashint_{B_{R_0}} \rho_0$ denotes the mean value of $\rho_0$ in $B_{R_0}(0)$.
Thus, for a given $\rho_0$, our results in Theorem \ref{main} are about global existence in case $\chi$ is bounded by the number in \eqref{chi_star}, while the blow-up result in \cite{W18} was derived for large $\chi$.
As mentioned earlier, global bounded solutions were constructed in \cite{Wink08} by assuming a type of logistic source, $\al - \bt \rho^r$,  which is given differently from ours, $\al \rho- \bt \rho^r$. It seems that the values of $\chi_*$ in the third and fourth cases of \eqref{chi_star} are essentially the same as the condition in \cite[Lemma 2.1]{Wink08}.
 In particular, in the case $\al=0$, two conditions are identical.
However, the critical value of $\chi$ remains open, so that blow-up or global existence is separated by this value.
\end{remark}

\begin{remark}
When $1<r<2$, 
our result reads that if $||\rho_0||_{L^\infty} \leq C_{\al,\bt,r}^{(2)}$, then the bounded weak solution globally exists as long as $\chi < C_{\al,\bt,r}^{(1)}$, independent of the size of $||\rho_0||_{L^\infty}$, where
$$
C_{\al,\bt,r}^{(1)}
:=\al^{\frac{2-r}{1-r}} \bt^{\frac{1}{r-1}}(2-r)^{\frac{2-r}{r-1}} (r-1), \qquad C_{\al,\bt,r}^{(2)}:=\big( \frac{\al}{\bt(2-r)} \big)^{\frac{1}{r-1}}.
$$ 
On the other hand, in the case that $||\rho_0||_{L^\infty}>C_{\al,\bt,r}^{(2)}$ and it is sufficiently large, $\chi$ is to be less than $\frac{\bt}{||\rho_0||_{L^\infty}^{2-r}}$ for the global existence.
\end{remark}

\begin{remark}\label{Dirichlet}
For variations of the elliptic equation \eqref{elliptic}, the proof of Theorem \ref{main} applies almost as well. For example, if we replace \eqref{elliptic} with 
\begin{equation}\label{Elliptic2}
  \begin{cases}
         -\Delta c= \rho& \text{in $\Om \times (0,T)$}, \\
\hfil	c|_{\p \Om}=0               & \text{on $\p \Om \times (0,T)$,} 
  \end{cases}
\end{equation}	
we obtain the same result as Theorem \ref{main}. 
To this end,
all we have to do is to replace $\cE_{1}$ with $\widetilde{\cE}_{1}$ defined as 
$\widetilde{\cE}_{1}(\rho)=                                 \int_{\Om} U \big( \rho(x) \big)  dx - \frac{\chi}{2} \int_\Om |\grad c(x)|^2 dx$
where $c$ depends on $\rho$ through $-\Delta c = \rho$ with $c=0$ on $\p \Om$.
\end{remark}

\begin{remark}\label{corollary}
With the same assumptions on $\Om, \rho_0$, and $\chi$ as in Theorem \ref{main}, we note that
the solution of \eqref{classical_KSL} becomes a classical solution in $[\delta,\infty) \times \Om$ for any $\del>0$.
In addition, if we further assume $\rho_0 \in (C^\infty \cap L^\infty)(\Om)$, then we have a unique classical solution in $[0, \infty) \times \OM$. 
\end{remark}

From now on, without loss of generality,  we will assume $\Lam \equiv 1$ since it does not affect our main result. We also suppose that $\Om$ is a strictly convex bounded domain. Throughout
the paper, we use the convention that $1/\alpha=+\infty$ when $\alpha=0$.

This paper is organized as follows:
Section \ref{section_pre} is prepared for the preliminaries. In Section \ref{section_splittingJKO}, we study two JKO schemes, namely Wasserstein and Fisher-Rao JKO schemes. In Section \ref{section_lemmas}, we construct approximate weak solutions. Section \ref{section_proofs} is devoted to providing the proof of Theorem \ref{main}.

\vspace{5mm}
\section{Preliminaries}\label{section_pre}

In this section, we deal with three metrics between positive measures and recall some properties on them.

We start with an introduction to \textit{Wasserstein metric} (see e.g. \cite{AGS08, S15, V03} for details). For consistency with \cite{CS18}, we use the half-quadratic cost
$c(x,y)=\frac{1}{2}|x-y|^2$.

\begin{definition}\label{W2_distance}
Let $\mu , \nu \in \cM^+(\Om)$ with equal masses. We define $2$-Wasserstein distance between $\mu$ and $\nu$ given by
\begin{equation}\label{optimal_plan}
W_2^2(\mu,\nu):= \min_{\gam \in \Gam(\mu,\nu)} \int_{\Om \times \Om} \frac{1}{2} |x-y|^2 d \gam(x,y)
\end{equation}
where $\Gam(\mu,\nu)$ denotes the subset of $\cM^+(\Om \times \Om)$ with $\mu$ as the first marginal and $\nu$ as the second marginal. 
\end{definition}
We recall several facts on $W_2$:
If $\mu$ is absolutely continuous with respect to $\cL^d$, then the equality in \eqref{optimal_plan} is attained for some $\gam=(id , T)_\# \mu$. In this case, $W_2^2(\mu,\nu)=\frac{1}{2}\int_\Om |T(x) -x|^2 d\mu(x)$, and we say that $T$ is an optimal transport map. Next, we introduce Kantorovich's duality, that is,
\begin{equation}\label{KantorovichDuality}
W_2^2(\mu,\nu) = \max_{(\vphi, \psi) } \Big( \int_\Om \vphi(x) ~d\mu(x) + \int_\Om \psi(y) ~d\nu(y) \Big)
\end{equation}
where the maximum is taken over all pairs $(\vphi,\psi) \in C(\OM) \times C(\OM)$ satisfying $\vphi(x)+ \psi(y) \leq \frac{1}{2} |x-y|^2$. Indeed, the maximum in \eqref{KantorovichDuality} is attained by a pair of $(\vphi, \vphi^c)$, where $\vphi^c(y):= \inf_{x \in \Om} \big(\frac{1}{2} |x-y|^2 - \vphi(x) \big)$. In this case, $\vphi$ is called a \textit{Kantorovich potential} from $\mu$ to $\nu$, and $\cK(\mu, \nu)$ denotes the set of all Kantorovich potentials from $\mu$ to $\nu$.
It is important to note that both $\vphi$ and $\vphi^c$ are Lipschitz. Since $\vphi$ is Lipschitz, it is differentiable a.e. Finally, we pay attention to the relationship between an optimal transport map $T$ from $\mu$ to $\nu$ and its Kantorovich potential $\vphi$. In fact,
the equality $T(x) = x- \grad \vphi(x)$ holds a.e.

\begin{theorem}[Benamou-Brenier formula]\label{BenamouBrenier}
Let $\rho_0, ~\rho_1 \in \cM^+(\Om)$ with $|\rho_0|=|\rho_1|$. Then
$$
W_2^2(\rho_0, \rho_1)= \min_{(\rho, v) \in \cA_{W_2}[\rho_0, \rho_1] } \int_0^1 \int_\Om  \frac{1}{2} |v_t|^2 d\rho_t dt,
$$
where the admissible set $ \cA_{W_2}[\rho_0, \rho_1]$
consists of curves $[0,1] \ni t  \mapsto (\rho_t, v_t) \in \cM^+(\Om) \times L^2(\Om, d\rho_t)^d$ such that $t \mapsto \rho_t$ is narrowly continuous with endpoints $\rho_0, \rho_1$ and solving the continuity equation
$$
\p_t \rho_t + \grad \cdot (\rho_t v_t) = 0
$$
in the sense of distributions, that is, for any $\vphi \in C_c^\infty\big(  (0,1)\times \overline{\Omega}  \big)$
\begin{equation}\nonumber
\int_0^1 \int_\Om \p_t \vphi d \rho_t dt + \int_0^1 \int_\Om \grad_x \vphi \cdot v_t d\rho_t dt=0.
\end{equation}
\end{theorem}

The \textit{Fisher-Rao distance} between $\rho_0$ and $\rho_1$, denoted by $\FR(\rho_0,\rho_1)$, is defined as follows:
\begin{definition}\label{FR_distance}
Let $\rho_0,~\rho_1 \in \cM^+(\Om)$. 
$$
\FR^2(\rho_0, \rho_1) := \min_{(\rho, r) \in \cA_{\FR}[\rho_0, \rho_1] } \int_0^1 \int_\Om |r_t|^2 d\rho_t dt,
$$
where the admissible set $ \cA_{\FR}[\rho_0, \rho_1]$
consists of curves $[0,1] \ni t  \mapsto (\rho_t, r_t) \in \cM^+(\Om) \times L^2(\Om, d\rho_t)$ such that $t \mapsto \rho_t$ is narrowly continuous with endpoints $\rho_0, \rho_1$ and solving 
$$
\p_t \rho_t = \rho_t r_t
$$
in the sense of distributions, that is, for any $\vphi \in C_c^\infty\big(  (0,1)\times \OM  \big)$
\begin{equation}\nonumber
\int_0^1 \int_\Om \p_t \vphi d \rho_t dt + \int_0^1 \int_\Om \vphi r_t d\rho_t dt=0.
\end{equation}
\end{definition}

\begin{remark}
If $\rho_0, \rho_1 \ll \lam$ for some $\lam \in \cM^+(\Om)$, the following formula is given \cite{GM17}:
\begin{equation}\label{FR_formula}
\FR^2(\rho_0, \rho_1)= 4 \int_\Om \bigg|  \sqrt{\frac{d\rho_0}{d \lam}}- \sqrt{\frac{d\rho_1}{d \lam}} \bigg|^2 d\lam,
\end{equation}
where $\frac{d\rho_i}{d \lam}$, $i=0,1$, denote Radon-Nikodym derivatives.
\end{remark}

The \textit{Wasserstein-Fisher-Rao distance} between $\rho_0$ and $\rho_1$, denoted by $\WFR(\rho_0,\rho_1)$, is defined as follows:
\begin{definition}\label{WFR_distance}
Let $\rho_0,~\rho_1 \in \cM^+(\Om)$. 
$$
\WFR^2(\rho_0, \rho_1) := \min_{(\rho, v,r) \in \cA_{\WFR}[\rho_0, \rho_1] } \int_0^1 \int_\Om  \left( \frac{1}{2} |v_t|^2 +|r_t|^2  \right) d\rho_t dt,
$$
where the admissible set $ \cA_{\WFR}[\rho_0, \rho_1]$
consists of curves $[0,1] \ni t  \mapsto (\rho_t, v_t,r_t) \in \cM^+(\Om) \times L^2(\Om, d\rho_t)^d \times L^2(\Om, d\rho_t)$ such that $t \mapsto \rho_t$ is narrowly continuous with endpoints $\rho_0, \rho_1$ and solving the continuity equation with source
$$
\p_t \rho_t + \grad \cdot (\rho_t v_t) = \rho_t r_t
$$
in the sense of distributions, that is, for any $\vphi \in C_c^\infty\big(  (0,1)\times \OM  \big)$
\begin{equation}\nonumber
\int_0^1 \int_\Om \p_t \vphi d \rho_t dt + 
\int_0^1 \int_\Om \big(\grad_x \vphi \cdot v_t + \vphi r_t  \big) d\rho_t dt=0.
\end{equation}

\end{definition}

\begin{remark}\label{WFR_metric_remark}
We leave some comments on the $\WFR$ metric: It was shown in \cite{KMV16} that
\begin{itemize}
\item $(\cM^+(\Om), \WFR)$ is a complete metric space. 
\item WFR metrizes 
narrow convergence on $\cM^+(\Omega)$.

\item $\WFR$ distance is lower semicontinuous with respect to the weak$^*$ convergence of measures.
\end{itemize}

\end{remark}

For $\mu \in \cM^+(\Om)$, we use the notation $| \mu |$ to denote $\mu(\Om)$. Using the three definitions above, it is not difficult to check that for any $\rho_0, \rho_1 \in \cM^+(\Om)$,
$$
\WFR(\rho_0,\rho_1) \leq W_2(\rho_0,\rho_1) \quad \textrm{whenever $|\rho_0|=|\rho_1|$,}
$$
and
\begin{equation}\label{WFR_leq_FR}
\WFR(\rho_0,\rho_1) \leq \FR(\rho_0,\rho_1). 
\end{equation}
Therefore, using the elementary inequality $(a+b)^2 \leq 2(a^2+b^2)$, we see that for all $\sigma \in \cM^+(\Om)$ with $|\rho_0| = |\sigma|$, 
\begin{equation}\label{pytagorean}
    \WFR^2(\rho_0, \rho_1) \leq 2 \big( W_2^2(\rho_0,\sigma) + \FR^2(\sigma, \rho_1 )   \big).
\end{equation}
We refer the reader to \cite{CPSV2018, CPSV18, GM17, KMV16,  LMS16, LMS18} for a comprehensive description of the above metrics.

\vspace{5mm}
Finally, let us briefly recall some classical results for \eqref{elliptic} and 
\eqref{Elliptic2}.  
Let $c$ be a solution of \eqref{elliptic} or \eqref{Elliptic2}.
It is well known that 
\begin{equation}\nonumber
||c||_{W^{2,p}(\Om)} \leq K ||\rho||_{L^p(\Om)}, \quad 1<p<\infty,
\end{equation}
for some $K>0$ (see e.g. \cite[Section 3 in Chapter 9]{K08} for Neumann boundary case and \cite[Chapter 9]{GT77} for Dirichlet boundary case).
For fixed $d$, 
 if $\rho \in L^\infty(\Om)$ and $p$ is large enough, by Sobolev inequality, there exist $\gam, K_1, K_2, K_3>0$ such that
 \begin{equation}\label{elliptic_estimate}
||c||_{C^{1,\gam}(\OM)} \leq K_1  ||c||_{W^{2,p}(\Om)}
\leq K_2 ||\rho||_{L^p(\Om)} \leq K_3 ||\rho||_{L^\infty(\Om)}.
\end{equation}

Thanks to the estimate \eqref{elliptic_estimate}, it is possible to check Remark \ref{Dirichlet} immediately. 
For the rest of this paper, we proceed with \eqref{elliptic}. \\

Finally, we recall an extension of the Aubin-Lions lemma proved in \cite{RS03}.
\begin{proposition}\label{ALRS}
On a Banach space $X$, let be given
\begin{enumerate}[label=(\alph*)]
\item a normal coercive integrand $\mathscr{F}: X \ra [0,\infty]$, i.e.,  $\mathscr{F}$ is lower semicontinuous, and its sublevels are relatively compact in $X$;

\item a pseudo-distance $g: X \times X \rightarrow [0,\infty]$, i.e., $g$ is lower semicontinuous, and $g(\rho, u)=0$ for any $\rho , u \in X$ with $\mathscr{F}(\rho) <\infty$, $\mathscr{F}(u) <\infty$ implies $\rho=u$.
\end{enumerate}
Let further $\mathscr{U}$ be a set of measurable functions $u : (0,T) \ra X$ with a fixed $T>0$. Under the hypotheses that 
\begin{equation}\label{time_H^1}
\sup_{u\in \mathscr{U}} \int_0^T \mathscr{F}(u(t)) dt <\infty,
\end{equation}
and
\begin{equation}\label{AL_limsup}
\lim_{h \da 0} \sup_{u \in  \mathscr{U}} \int_0^{T-h} g(u(t+h) , u(t))dt =0,
\end{equation}
$\mathscr{U}$ contains an infinite sequence $(u_n)_{n \in \N}$ that converges in measure (with respect to $t \in (0,T)$) to a limit $u^*:(0,T) \ra X$. 
\end{proposition}

\begin{remark}
At the conclusion of the above proposition, the convergence in measure means that
$$
\lim_{n \ra \infty} 
\Big|  \{t \in (0,T) : 
|| u_n(t) - u^* (t)   ||_{X} \geq \ep   \}  \Big| =0 \qquad \forall \ep>0.
$$
\end{remark}

\vspace{5mm}
\section{Splitting JKO scheme}\label{section_splittingJKO}
Motivated by \cite{CS18, GLM19, GM17}, we analyze \eqref{KSL} by exploiting two gradient flow structures \eqref{GF1} and \eqref{GF2}. To be more precise, we alternately use the Wasserstein JKO scheme / Fisher-Rao JKO scheme.

\subsection{Wasserstein JKO step}
We recall the results introduced in \cite[Theorem 1 and Subsection 2.1]{CS18} in our setting, which is presented with the following minimization problem: For fixed $\T>0$, and for $g \in L^\infty(\Om) \cap \cM^+(\Om)$,
\begin{equation}\label{argmin}
\rho \in \argmin_{ \mu \in \cS}
\Big\{ \cE_{1}(\mu) +  \frac{W_2^2(g,\mu)}{\T} \Big\},
\end{equation}
where $\cS= \{   \mu \in \cM^+(\Om) \cap    L^\infty(\Om) : \mu \leq \frac{1}{\T \chi} , ~|\mu|=|g | \} $.

Thanks to \cite[Theorem 1 and Subsection 2.1]{CS18}, the following proposition holds. Although the proof is essentially the same, we provide its proof in the Appendix for the convenience of the reader.

\begin{proposition}\label{CS_theorem}
Suppose that $U$ satisfies Assumption \ref{assumptionsUF}. The following statements hold.
\begin{itemize}
\item[\emph{(i)}] 
There exists at least one minimizer of the scheme in \eqref{argmin}.

\item[\emph{(ii)}] 
Let $\overline{B}:=\frac{1}{\T \chi}$.
If $\inf_{t>0} tU''(t)>0$ and $\rho$ is a minimizer of \eqref{argmin}, there exist a continuous function $p$ with $p \geq 0$ and a constant $l$ such that  $\rho$ satisfies
\begin{equation}\label{Santambrogio!}
    U'(\rho) - \chi c[\rho] + \frac{\vphi}{\T} +p=l,  \qquad (\overline{B} -\rho)p=0,
\end{equation}
where $c[\rho]$ is a solution in \eqref{elliptic}, and 
$\vphi$ is a Kantorovich potential from $\rho$ to $g$, i.e., $\vphi \in \cK(\rho, g)$. In particular, $U'(\rho), c[\rho]$, and $\vphi$ are Lipschitz on $\OM$. Moreover, $\inf_\Om \rho>0$ and
$\rho$ is Lipschitz on $\OM$.
\item[\emph{(iii)}] 
If $\Om$ is strictly convex, $\inf_{t>0} tU''(t)>0$ and $\log g \in C^{0,a}(\OM)$, then 
for every $\lam>1$ there exists a constant $c_0=c_0(\lam, \chi, d)$ such that, if $\T ||g||_{L^\infty} \leq c_0$ then any minimizer of \eqref{argmin}, say $\rho$, satisfies
\begin{equation}\label{CS_estimate}
|| \rho ||_{L^\infty} \leq \frac{||g||_{L^\infty} }{ 1- \lam \T \chi ||g||_{L^\infty}}.
\end{equation}

\end{itemize}
\end{proposition}

\begin{remark}
Proposition 3.1 is only used in construction steps for approximate solutions by JKO scheme,
and therefore $\inf_{t>0} tU''(t)>0$ is not needed in Theorem \ref{main}.
\end{remark}

\subsection{Fisher-Rao JKO step}

Given $\rho \in \cM^+ (\Om)$, we consider
\begin{equation}\label{argmin2}
\hrho \in \argmin_{\mu \in \cM^+(\Om)} 
\Big\{ \cE_2(\mu) +  \frac{\FR^2(\rho, \mu)}{2\T} \Big\}.
\end{equation}

We clarify the role of the Fisher-Rao variational formulation in \eqref{argmin2}. The minimization problem \eqref{argmin2} is used as a variational construction of the reaction substep. More precisely, it allows us to obtain the Euler-Lagrange equation and the $L^\infty$-damping estimate. Although $\cE_2$ decreases during each isolated Fisher--Rao reaction substep, we do not claim any monotonicity of $\cE_2$ along the full splitting scheme. Indeed, the Wasserstein substep is not designed to decrease $\cE_2$, and no monotonicity property of $\cE_2$ for the combined scheme is asserted or used. Thus, the Fisher-Rao viewpoint is used here as a technical variational device for the reaction step, rather than as an additional energy dissipation structure for the full discrete evolution.

The second minimization problem \eqref{argmin2} was considered in \cite{GLM19} (although the authors only treated $F(\rho)= \frac{\rho^m}{m-1}$, $m>1$). In this paper, we proceed with $F$ under the assumptions (A1)-(A3). In the absence of a precise reference, we prove the existence of a minimizer in the Fisher-Rao step.

\begin{lemma}\label{existence_FR}
Suppose that \emph{(A1)-(A2)} hold. Let $\rho \in \cM^+(\Om)$ with $\cE_2(\rho)<\infty$. Then, for any $\T>0$ there exists a minimizer of the following functional:
\begin{equation}\label{minimize_FisherRao}
\mu \in  \cM^+(\Om)  \mapsto \cE_2(\mu) + \frac{\emph{FR}^2(\rho, \mu)}{2\T}.
\end{equation}
\end{lemma}

\begin{proof}
We note that $F$ is bounded below.
Considering the minimizing sequence of the functional \eqref{minimize_FisherRao}, there exists a sequence $(\rho_n)$ such that
$$
\inf_{\mu \in \cM^+(\Om)} \Big( \cE_2(\mu) + \frac{\FR^2(\rho, \mu)}{2\T} \Big) 
= \lim_{n \ra \infty}  \cE_2(\rho_n) + \frac{\FR^2(\rho, \rho_n)}{2\T},
$$
and $\cE_2(\rho_n) + \FR^2(\rho, \rho_n)/(2\T)$ is uniformly bounded in $n$. Using that $\cE_2$ is bounded below,
\begin{equation}\label{bound1}
|\cE_2(\rho_n)| \leq C 
\end{equation}
for some $C>0$. 
Due to (A1)-(A2), we easily get
$$
\cE_2(\rho_n)=\int_\Om F(\rho_n) \geq F(0) |\Om| + \int_\Om \frac{\bt}{r} \rho_n^r - \al \rho_n.
$$
Applying Young's inequality and H\"{o}lder's inequality, we have
$$
\sup_n ||\rho_n||_{L^r} < \infty,
$$
where we used \eqref{bound1}.
Hence, by weak compactness, up to a subsequence, we have
\begin{equation}\label{r_weak}
\rho_n \wra \hrho \quad \textrm{in} 	\quad L^r(\Om),
\end{equation}

On the other hand, we note that $\cE_2$ is lower semicontinuous with respect to weak convergence in $L^r$ ($r>1$), that is,
\begin{equation}\label{E2_LSC}
\liminf  \cE_2(\mu_n) \geq \cE_2(\mu)
\end{equation}
whenever $\mu_n \wra \mu$ in $L^r(\Om)$.
In fact, this follows from the fact that the functional $\cE_2$ is weak$^*$ lower semicontinuous (see e.g. \cite[Theorem 3.4.1]{B89}).

We also investigate some weak* lower semicontinuity of $\FR^2$. It is known that
\begin{equation}\label{FR_LSC}
\liminf  \FR^2(\mu_n, \nu_n) \geq  \FR^2(\mu, \nu)
\end{equation}
whenever $\mu_n \wsra \mu$ and $\nu_n \wsra \nu$ (see e.g. \cite[Lemma 3.1.6 and Theorem 3.4.1]{B89})

Therefore, we conclude
$$
\inf_{\mu \in \cM^+(\Om)} \Big( \cE_2(\mu) + \frac{\FR^2(\rho, \mu)}{2\T} \Big) 
= \liminf_{n \ra \infty}  \Big( \cE_2(\rho_n) + \frac{\FR^2(\rho, \rho_n)}{2\T} \Big)
\geq \cE_2(\hrho) + \frac{\FR^2(\rho, \hrho)}{2\T},
$$
where we used \eqref{r_weak}  and the lower semicontinuity in \eqref{E2_LSC} and \eqref{FR_LSC}. Thus, $\hrho$ is a minimizer of \eqref{minimize_FisherRao}.
\end{proof}

\vspace{3mm}
We recall an Euler-Lagrange equation for \eqref{argmin2}, introduced in \cite[Subsection 4.2]{GM17}. 
In the absence of a precise reference with sufficient conditions, we prove the following lemma for clarity, following the arguments in \cite[Subsection 4.2]{GM17}.

\begin{lemma}\label{lemma_EL2}
Let $\rho \in \cM^+(\Om)$ with $\cE_2(\rho)<\infty$. Suppose that $F$ satisfies the assumptions \emph{(A1)}-\emph{(A3)}.
If $\hrho$ is a minimizer obtained from \eqref{argmin2} for the given datum $\rho$, then
\begin{equation}\label{Gallouet}
\int_\Om (\sqrt{\hrho} - \sqrt{\rho}) \sqrt{\hrho} \psi = -\frac{\T}{2} \int_\Om F'(\hrho) \hrho \psi
 \qquad   \forall \psi \in L^\infty(\Om).
\end{equation}
\end{lemma}

\begin{proof}
Let $\hrho$ be a minimizer of the problem \eqref{argmin2}. Given $0<\ep<1$ and $\psi \in L^\infty(\Om)$, define $\hrho^\ep = \hrho e^{\ep \psi}$.
Due to the minimality of $\hrho$, we have
$$
\cE_2(\hrho) + \frac{\FR^2(\hrho, \rho)}{2\T}
\leq \cE_2(\hrho^\ep) + \frac{\FR^2(\hrho^\ep, \rho)}{2\T}.
$$
The assumption (A1) implies that for all $\ep \in (0,1)$,
$$
\bigg| \frac{F(\hrho e^{\ep \psi}) - F(\hrho)}{\ep}  \bigg|
\leq \hrho e^{|\psi|}  |\psi| \big( |F'(0)| + |F'(\hrho e^{|\psi|})|  \big).
$$
Since $|\psi(x)| \leq B$ for some $B>0$, by the assumption (A3) and the dominated convergence theorem, we see that $\big(\cE_2(\hrho^\ep) -\cE_2(\hrho) \big)/\ep \ra \int_\Om F'(\hrho) \hrho \psi$ as $\ep \downarrow 0$.
Using \eqref{FR_formula} with $\lam=\cL^d \big|_{\Om}$, we see that $\big( \FR^2(\hrho^\ep, \rho) - \FR^2(\hrho, \rho) \big) /\ep \ra 
4 \int_\Om \sqrt{\hrho} (\sqrt{\hrho} - \sqrt{\rho} ) \psi $. 
Replacing $\psi$ with $-\psi$ gives \eqref{Gallouet}.
\end{proof}

\vspace{3mm}
Throughout the paper, for any $M>0$ we define
\begin{equation}\label{definition_eta}
\eta_{_M}:=\Big( \frac{\al + M}{\bt} \Big)^{\frac{1}{r-1} }.
\end{equation}  
We will denote by $M_* > 0$ a constant whose value will be specified in Lemma \ref{calculus}, and write \(\eta_{M_*}\) for the corresponding quantity.
In particular, the $L^\infty$-norm of the minimizers arising in the JKO scheme can be bounded in terms of $||\rho_0||_{L^\infty}$ and $\eta_{M_*}$ (see Proposition \ref{surprise} for details).

We often use the notation $\eta$ instead of $\eta_{_M}$  if there is no confusion. Moreover, for notational simplicity, we often use $|| \cdot ||_{\infty}$ instead of $|| \cdot ||_{L^\infty}$

The following lemma tells us how much the $L^\infty$ norm decreases when passing through the Fisher-Rao step.
\begin{lemma}\label{essential_first_estimate}
Suppose that \emph{(A1)-(A3)} hold. Let $\rho \in L^\infty(\Om) \cap \cM^+(\Om)$ and let $\hrho$ be a minimizer obtained by \eqref{argmin2} with the given datum $\rho$. If $M>0$ and  
if $\al \T<1$, then we have
\begin{equation}\label{essential_key_estimate}
||\hrho||_{L^\infty} \leq \max(\frac{ ||\rho||_{L^\infty} }{ 1+\T M } ,\eta)
\end{equation}
with $\eta=\eta_M$ specified in \eqref{definition_eta}.
\end{lemma}

\begin{proof}
If $||\hrho||_{\infty} \leq \eta$, there is nothing to prove.
Hence, we assume 
\begin{equation}\label{good_idea}
    ||\hrho||_{\infty}> \eta=\Big( \frac{\al + M}{\bt} \Big)^{\frac{1}{r-1} } \quad \textrm{or} \quad ||\hrho||_\infty^{r-1}>\frac{\al +M}{\bt}.
\end{equation}

\noindent The Euler-Lagrange equation \eqref{Gallouet} implies
\begin{equation}\label{Gallouet_pointwise}
\sqrt{\hrho}- \sqrt{\rho} = -\frac{\T}{2} \sqrt{\hrho}F'(\hrho)
\qquad  \hrho \textrm{-a.e.}
\end{equation}
Clearly, 
\begin{equation}\label{rho_half_expression}
    \rho =\hrho+ 
    \T \hrho F'(\hrho) + \frac{\T^2}{4} \hrho(F'(\hrho))^2  
    \qquad  \hrho \textrm{-a.e.,}
\end{equation}
which yields for a.e. $x \in \Om$,
\begin{align}
 \rho \geq \hrho+ 
    \T \hrho F'(\hrho) + \frac{\T^2}{4} \hrho(F'(\hrho))^2  
 \geq
    \hrho  + \T \hrho (\bt \hrho^{r-1} -\al)  \geq (1-\al \T)\hrho \geq0.                
\label{implicit_form_positivity}
\end{align}

\noindent Hence,
\begin{align}
    ||\rho||_{\infty} 
&\geq    ||\hrho  + \T \hrho (\bt \hrho^{r-1} -\al)||_{\infty}
= ||\hrho ||_{\infty} + \T ||\hrho ||_{\infty}( \bt||\hrho ||_{\infty}^{r-1}-\al) 
\nn \\
&\geq ||\hrho ||_{\infty} + \T M ||\hrho ||_{\infty} 
= (1+\T M )||\hrho ||_{\infty} \nn
\end{align}
where we have used \eqref{implicit_form_positivity} and \eqref{good_idea}.
Thus, we obtain \eqref{essential_key_estimate}.
\end{proof}

\vspace{3mm}
The following lemma shows that the minimizer of the Fisher-Rao step preserves the support, the Lipschitz condition, and the strict positivity of the minimizer of the Wasserstein step.
\begin{lemma}
\label{preserved_supports}
Suppose that Assumption \ref{assumptionsUF} holds.
Given $\rho_0 \in L^\infty(\Om)$, let $\rho$ be a minimizer obtained by \eqref{argmin} with the initial datum $\rho_0$. Similarly, let $\hrho$ be a minimizer given by \eqref{argmin2} with the given datum $\rho$. 
Then
\begin{itemize}
\item[\emph{(i)}] 
For $\T \in (0, \frac{2}{\al})$,
$$
    \emph{supp}  (\rho) = \emph{supp}  (\hrho).
$$

\item[\emph{(ii)}] 
If $\T \in (0, \frac{2}{\al})$ and if $\inf_{t>0} tU''(t)>0$, then
$$
\emph{supp}  (\rho) = \emph{supp}  (\hrho) = \Om.
$$
Moreover,
if $\T \in (0,\frac{1}{2\al})$, then $\hrho$ is Lipschitz and $\inf_\Om \hrho >0$.
\end{itemize}
\end{lemma}

\begin{proof}
(i)
Since $F'\geq -\al$, by using \eqref{Gallouet_pointwise}, we have
\begin{equation*}
\sqrt{\rho} = \sqrt{\hrho} \big(1 +\frac{\T}{2} F'(\hrho) \big) \geq  \sqrt{\hrho} \big(1 -\frac{\al \T}{2}  \big)
\qquad  \hrho \textrm{-a.e.,}
\end{equation*}
which implies $\textrm{supp} (\hrho) \subset \textrm{supp} (\rho)$.
To prove that $\textrm{supp} (\hrho) = \textrm{supp} (\rho)$ by contradiction, we assume $\textrm{supp} (\hrho) \subsetneq \textrm{supp} (\rho)$, and hence we can find $z\in \Om$ so that
$$
z \in \textrm{supp} (\rho) \backslash \textrm{supp} (\hrho) \neq \emptyset.
$$

\noindent
Then there exists $r^*>0$ such that $\rho(B_r(z)) >0$ and  $\hrho(B_r(z)) =0$ for all $ 0<r<r^*$. If we define $\hrho^{\ep}:=\hrho + \ep \textbf{1}_{B_r(z)}$ with $\ep>0$,
then by the minimality of $\hrho$ we have
\begin{equation}\label{FR_minimality}
\cE_2 (\hrho) +\frac{1}{2\T} \FR^2(\rho , \hrho)
\leq
\cE_2 (\hrho^{\ep}) + \frac{1}{2\T} \FR^2(\rho, \hrho^{\ep}).
\end{equation}
Note that
$$
\frac{1}{4} \big(\FR^2(\rho, \hrho^{\ep}) - \FR^2(\rho, \hrho)\big)
=\int_{B_r(z)} |\sqrt{\rho} - \sqrt{\hrho^{\ep}}|^2 -  |\sqrt{\rho} - \sqrt{\hrho}|^2
$$ 
$$
=\int_{B_r(z)} |\sqrt{\rho} - \sqrt{\ep}|^2
- |\sqrt{\rho} |^2 = \int_{B_r(z)} -2\sqrt{\ep \rho} + \ep.
$$
From \eqref{FR_minimality}, we obtain
$$
0 \leq \int_{B_r(z)} F(\hrho^\ep) - F(\hrho) +\frac{2}{\T} \int_{B_r(z)} \Big(-2\sqrt{\ep \rho} + \ep \Big).  
$$
Dividing by $\sqrt{\ep}$ and letting $\ep \downarrow 0$, we get a contradiction because $F'(0+)$ exists and $\int_{B_r(z)} \sqrt{\rho} >0$.
Thus, we obtain the desired result.

(ii) Following the proof of \cite[Lemma 8.6]{S15}, we can see that \textrm{supp}$(\rho) = \Om$. The assumption $\inf_{t>0} tU''(t)>0$ is necessary here to require $U'(0+)=-\infty$.  Due to (i), we conclude that supp$(\hrho)= \Om$.  

To prove the second statement, from (ii) of Proposition \ref{CS_theorem}, we recall that
$\rho$ is Lipschitz and $\essinf_\Om \rho >0$. 
Using Lemma \ref{lemma_EL2} and the fact that supp$(\hrho)=\Om$, we obtain
\begin{equation}\label{Gallouet_a.e.}
\sqrt{\hrho}- \sqrt{\rho} = -\frac{\T}{2} \sqrt{\hrho}F'(\hrho) \qquad \textrm{a.e. $x \in \Om$.}
\end{equation}
Similar to \eqref{rho_half_expression}, we get
\begin{equation}\label{Gallouet_square_a.e.}
\rho = J_\T(\hrho), \quad  \textrm{a.e. } x \in \Om,
\end{equation}
where $J_\T(s) := s+ \T s F'(s) + \frac{\T^2}{4}s (F'(s))^2$, $s\geq0$. Using (A1) and (A2), we see that $J_\T' (s) \geq 1-\al \T +\T s F''(s)(1-\frac{\al \T}{2})$. Hence, if $0<\T< \frac{1}{2\al}$, then
$\inf_{s > 0} J_\T' \geq 1/2$. In this case, since $J_\T$ is strictly increasing,  we can assume $\hrho$ is continuous (by redefining $\hrho := J_\T^{-1}(\rho))$. 
Therefore, we have
$$
||\hrho||_{\textrm{Lip}} \leq \frac{1}{\inf_{s>0} J_\T'} ||\rho||_{\textrm{Lip}}
 \leq 2 ||\rho||_{\textrm{Lip}}.
$$

It thus remains to show $\inf_\Om \hrho >0$. Since supp$(\hrho)=\Om$ and $\hrho$ is Lipschitz, from \eqref{Gallouet_square_a.e.}, we have
\begin{equation}\label{nth_EL_squaring}
    \rho =\hrho+ 
    \T \hrho F'(\hrho) + \frac{\T^2}{4} \hrho (F'(\hrho))^2  \quad \textrm{for all $x\in \Om$}.
\end{equation}
If $\inf_\Om \hrho =0$, then there exists $x_0 \in \OM$ such that $\hrho(x_0)=0$ due to the regularity of $\hrho$. Then, \eqref{nth_EL_squaring} implies a contradiction to the fact that  $\essinf_\Om \rho >0$ (if $x_0 \in \p \Om$, choose a sequence $(x_k)$ converging to $x_0$), and thus
we obtain the desired result.
\end{proof}

\begin{remark}
Let $\rho,\hrho$ be the minimizers in (ii) of Lemma \ref{preserved_supports}. Then it is possible to obtain 
\begin{equation}\label{Gallouet!}
\int_\Om (\hrho - \rho) \phi = -\frac{\T}{2} \int_\Om F'(\hrho)\sqrt{\hrho} \big(\sqrt{\hrho} + \sqrt{\rho } \big) \phi \qquad   \forall \phi \in C^\infty(\OM),
\end{equation}
which is a variant of Lemma \ref{lemma_EL2}. Indeed, since $\rho$ and $\hrho$ are Lipschitz, multiplying \eqref{Gallouet_a.e.} by $\sqrt{\hrho}+ \sqrt{\rho}$
and taking $\phi \in C^\infty(\OM)$ as test functions, we get \eqref{Gallouet!}.

\end{remark}

\vspace{5mm}
\section{Approximate weak solutions: Construction and $L^\infty$-estimates}\label{section_lemmas}

In this section, we construct an approximate weak solution via splitting JKO scheme. More precisely, we alternately repeat the Wasserstein and Fisher-Rao JKO steps;
$$
\rho_0 \xrightarrow{\text{$W_2^2$}} \rho_{1/2}^\T 
\xrightarrow{\text{$\FR^2$}} \rho_1^\T
\xrightarrow{\text{$W_2^2$}} \rho_{3/2}^\T
\xrightarrow{\text{$\FR^2$}} \rho_2^\T 
\xrightarrow{\text{$W_2^2$}}...
\xrightarrow{\text{$\FR^2$}} \rho_n^\T
\xrightarrow{\text{$W_2^2$}}
\rho_\half
\xrightarrow{\text{$\FR^2$}} \rho_\one
\xrightarrow{\text{$W_2^2$}}...
$$
i.e., we iteratively define
\begin{equation}\label{xrightarrow}
\rho_\half \in \argmin_{\mu \in \cS_n} \Big\{ \cE_{1}(\mu) +  \frac{W_2^2(\rho_n^\T,\mu)}{\T} \Big\}, \quad
\rho_\one \in \argmin_{\mu \in \cM^+(\Om)} \Big\{ \cE_{2}(\mu) +  \frac{\FR^2(\rho_\half,\mu)}{2\T} \Big\},
\end{equation}
where $\cS_n= \{   \mu \in \cM^+(\Om) \cap    L^\infty(\Om) : \mu \leq \frac{1}{\T \chi} , ~|\mu|=|\rho_n^\T | \} $.

Then, for each $n\geq0$, we define
\begin{align}
\tilde{\rho}^\T(0)&:=\rho_0, 
\qquad \quad     \rho^\T(0):=\rho_0,      \nn \\
\tilde{\rho}^\T(t)&:=\rho_{n+1/2}^\T, \quad \rho^\T(t):=\rho_{n+1}^\T, \quad \textrm{whenever} \quad t \in \big( n\T ,(n+1)\T].           \label{two_curves_definition}
\end{align}

It will be shown that the curves $\tilde{\rho}^\T(t), \rho^\T(t)$ approximate weak solutions of \eqref{KSL}.
To this end, we start by analyzing the minimizers in \eqref{xrightarrow}.

\vspace{3mm}
The following lemma shows that the splitting schemes are well-defined.
To repeat the scheme \eqref{argmin} and \eqref{argmin2} alternately and infinitely, we need the following $L^1$-estimate of $\rho_n^\T$ because of the density constraint of \eqref{argmin}.

\begin{lemma}\label{L^1_bound}
Suppose that Assumption \ref{assumptionsUF} holds. Let $\rho_0 \in L^\infty(\Om)$. Then, there exists $\xi>0$ \emph{(}depending on $\al, \bt, r, |\Om|$, $||\rho_0||_{L^1(\Om)}$\emph{)} such that if $0<\T<\frac{|\Om|}{\xi \chi}$, then the minimizers $\rho_\half$ and $\rho_\one$ in \eqref{xrightarrow} exist for each $n \in \N \cup \{ 0\}$, and
$$
||\rho_n^\T||_{L^1(\Om)} \leq \xi.
$$
\end{lemma}

\begin{proof}
If $0<\T <   \frac{|\Om|}{\chi ||\rho_0||_{L^1}}$, from (i) in Proposition \ref{CS_theorem}, we admit a minimizer of \eqref{argmin} for the initial datum $\rho_0$, say $\rho_{1/2}$. Here, we need the condition on $\T$ because of the constraints in \eqref{argmin}. 
Also, by Lemma \ref{existence_FR}, we have $\rho_1$, which is a minimizer of  \eqref{argmin2} for the initial datum $\rho_{1/2}$. 
Due to \eqref{rho_half_expression}, we have
$$
\int_\Om \rho_0 = \int_\Om \rho_{1/2} \geq  \int_\Om \rho_1 +\int_\Om \T \rho_1 F'(\rho_1) \geq \int_\Om \rho_1 + \T \Big(\bt \int_\Om \rho_1^r - \al \int_\Om \rho_1 \Big).
$$
Using H\"{o}lder's inequality, we have 
\begin{equation}\label{L1control1}
\int_\Om \rho_{0} \geq  
\int_\Om \rho_1 +\T \Big( \bt k_0^{-r} \big(\int_\Om \rho_1 \big)^r - \al \int_\Om \rho_1 \Big) 
\end{equation}
with $k_0 = |\Om|^{1/r'}$, where $r'$ is the H\"{o}lder conjugate of $r$. For $\ep>0$ small enough, applying Young's inequality, we have
\begin{equation}\label{L1control2}
\int_\Om \rho_1 \leq \ep \Big(\int_\Om \rho_1 \Big)^r + C_\ep,
\end{equation}
where $C_\ep :=(r\ep)^{\frac{-1}{r-1}}/r'$.  Combining \eqref{L1control1} and \eqref{L1control2}, we obtain
$$
||\rho_0||_{L^1(\Om)} \geq ||\rho_1||_{L^1(\Om)} +\T A_\ep ||\rho_1||_{L^1(\Om)}  - \T B_\ep,
$$
where $A_\ep=\bt k_0^{-r} \ep^{-1} -\al$ and $B_\ep= \bt k_0^{-r} \ep^{-1} C_\ep$. Then, 
\begin{equation}\label{L^1_first_estimate}
||\rho_1||_{L^1(\Om)} \leq
\frac{||\rho_0||_{L^1(\Om)} + \T B_\ep }{1+ \T A_\ep} = \Big(||\rho_0||_{L^1(\Om)} -\frac{B_\ep}{A_\ep} \Big)\frac{1}{1+A_\ep \T} + \frac{B_\ep}{A_\ep}.
\end{equation}
Fix $\ep>0$ small so that $A_\ep$ is positive, and we set
$$
\xi := \max \big( ||\rho_0||_{L^1(\Om)},  \frac{B_\ep}{A_\ep}  \big)
\quad \textrm{and} \quad
\bar{\T}:= \frac{|\Om|}{\xi \chi}.
$$
Since $A_\ep>0$, from \eqref{L^1_first_estimate}, 
we know that if $0<\T <\bar{\T} ~(\leq \frac{|\Om|}{\chi ||\rho_0||_{L^1(\Om)}})$ then 
\begin{equation}\label{L^1_bounded!}
||\rho_1||_{L^1(\Om)} \leq \xi.
\end{equation}

Using Lemma \ref{essential_first_estimate}, we know $\rho_1 \in L^\infty(\Om)$. Moreover, a simple perturbation argument shows $||\rho_1||_{L^1}$ cannot be zero (suppose $\rho_1=0$ a.e., and define $\rho_1^{\ep}=\ep$ on $\Om$, and use the minimality of $\rho_1$). 
So we can repeat the above argument to obtain \eqref{L^1_first_estimate} for $\rho_2$. More precisely, using again (i) in Proposition \ref{CS_theorem} and Lemma \ref{existence_FR}, we see that
for $0< \T < \frac{|\Om|}{\chi} \min(||\rho_0||_{L^1}^{-1}, ||\rho_1||_{L^1}^{-1})$, $\rho_{3/2}$ and $\rho_2$ exist,  and
$$
||\rho_2||_{L^1(\Om)} \leq
\frac{||\rho_1||_{L^1(\Om)} + \T B_{\ep} }{1+ \T A_{\ep}}.
$$
Solving this recurrence relation, we get
$$
||\rho_2||_{L^1(\Om)} \leq
(||\rho_0||_{L^1(\Om)} -\frac{B_\ep}{A_\ep}) \big(  \frac{1}{1+A_\ep \T} \big)^2 + \frac{B_\ep}{A_\ep}.
$$
Similar to \eqref{L^1_bounded!}, if $0<\T  < \bar{\T} ~\big(\leq \frac{|\Om|}{\chi} \min(||\rho_0||_{L^1}^{-1}, ||\rho_1||_{L^1}^{-1})  \big)$, then 
$$
||\rho_2||_{L^1(\Om)} \leq \xi.
$$
Moreover, we know that $\rho_2 \in L^\infty(\Om)$ by  Lemma \ref{essential_first_estimate}. 
We obtain the desired result by induction.
\end{proof}

\vspace{3mm}
We turn our attention to getting $L^\infty$-estimates of the minimizers $\rho_\half$ and $\rho_\one$.
To this end, the following lemma plays a key role.

\begin{lemma}\label{calculus}
Let $\rho_0 \in L^\infty(\Om)$ and let $M>0$ be a parameter. Suppose that $0<\chi < \chi_*$, where $\chi_*$ is the number stated in \eqref{chi_star}, and let $\lam>1$ with $\lam  \chi < \chi_*$.
Define two functions of $\T$, namely
$$
\theta_1(\T):=(1+\T M)(||\rho_0||_{L^\infty}^{-1} - \lam \chi \T) \quad and \quad \theta_2(\T) := (1+\T M)(\eta_{_M}^{-1} - \lam \chi \T)
$$ 
which are defined on $[0,\infty)$.
Then, we can find $M$, namely $M_*$, so that $\theta_i'(0)>0$, $i=1,2$. As a consequence, there exists $\T_*>0$ such that  $\theta_i(t) \geq \theta_i(0)$, $i=1,2$, whenever $t\in (0,\T_*)$.

\end{lemma}
\begin{remark}
In the proof of Lemma \ref{calculus}, the value $\chi_*$ in \eqref{chi_star} is exactly determined. In fact,
$$
\chi_*=\sup_{M>0} \Big( \min( M \eta_M^{-1}, M ||\rho_0||_{L^\infty}^{-1} ) \Big).
$$

\noindent
We also note that
 $M_*$ may be chosen as follows:
\begin{equation*}
M_*
:=
\begin{cases}
 1+   \inf \cS  & \text{if $r \geq 2$,} \\
 \bt ||\rho_0||_{L^\infty}^{r-1}  - \al  & \text{if $1<r<2$ \quad and \quad $||\rho_0||_{L^\infty} > \big( \frac{\al}{\bt(2-r)} \big)^{\frac{1}{r-1}} $,} \\
  
\frac{\al(r-1)}{2-r}  & \text{if $1<r<2$ \quad and \quad $||\rho_0||_{L^\infty} \leq \big( \frac{\al}{\bt(2-r)} \big)^{\frac{1}{r-1}}   $},
     \end{cases}
    \end{equation*}
where $\cS:= \{M>0 : M ||\rho_0||_{\infty}^{-1} - \lam \chi>0 \textrm{~and~} M\eta_M^{-1}- \lam \chi>0 \}.$
\end{remark}

\begin{proof}
We want to find $M>0$ so that 
\begin{equation}\label{possible.}
\theta_1'(0)=M||\rho_0||_\infty^{-1} - \lam \chi>0 \quad 
\end{equation} 
and
\begin{equation}\label{possible!}
\theta_2'(0)=M \eta_{_M}^{-1} - \lam \chi >0.
\end{equation} 
It is straightforward to choose $M$ so that \eqref{possible.} holds by choosing $M$ large enough. 
Note that \eqref{possible!} is equivalent to
$$
\frac{M \bt^{\frac{1}{r-1}} }{\big( \al + M \big)^{\frac{1}{r-1} }}:=f(M) > \lam \chi.
$$

To find $M$ satisfying \eqref{possible!}, we consider three cases of the range of $r$:

(Case $r>2$)  
Notice that $f(M) \ra \infty$ as $M \ra \infty$ because $r>2$. 
Hence, it is always possible to find $M_*$ satisfying \eqref{possible.} and \eqref{possible!}.

(Case $r=2$) Let us first consider the case $\al>0$. By our assumption in this case, we have
$$
\lam \chi < \bt = \chi_*.
$$ 
Since $f(M) \uparrow \bt$ as $M \ra \infty$, there exists a large $M$ such that $\lam \chi < f(M)$, and hence
we can find $M_*$ satisfying \eqref{possible.} and \eqref{possible!}.
The case  $\al=0$ follows directly, because $f(M) \equiv \bt$.

(Case $1<r<2$)  Unlike the above two cases, since $f(M) \ra 0$ as $M \ra \infty$, we avoid growing $M$. 
Recalling the definition of $\chi_*$, in the case that 
\begin{equation}\label{critical_region1}
||\rho_0||_{L^\infty(\Om)}^{-1} <  \bigg( \frac{\al}{\bt(2-r)} \bigg)^{\frac{1}{1-r}},
\end{equation}
we know
\begin{equation}\label{chi_star_determined}
\lam \chi < \chi_*=||\rho_0||_{L^\infty}^{-1} (\bt ||\rho_0||_{L^\infty}^{r-1}  - \al).
\end{equation}
If we set $M_*:=\bt ||\rho_0||_{L^\infty}^{r-1}  - \al$, then $M_* ||\rho_0||_{\infty}^{-1} = M_* \eta_{M_*}^{-1} = \chi_*$, and therefore \eqref{possible.} and \eqref{possible!} hold due to 
\eqref{chi_star_determined}; Such a choice of $M_*$ is optimal in the sense that the range of $\chi$ is maximized.
Indeed, from \eqref{possible.} and \eqref{possible!}, $\chi_*$ is determined as follows: 
\begin{equation}\label{maxmin_fg}
\chi_* = \max_{M>0} \Big( \min( f(M),g(M)) \Big),
\end{equation}
where $g(M):= M||\rho_0||_{L^\infty}^{-1}$. 
Notice that \eqref{critical_region1} is equivalent to 
$$
g'(0)<\frac{f(M_f)}{M_f},
$$
where $M_f:=\frac{\al(r-1)}{2-r}$ is a unique maximizer of $f$.
Since $M_*=\bt ||\rho_0||_{L^\infty}^{r-1}  - \al$ is a unique positive solution of the equation $f=g$ and it is a maximizer of the function $\min(f,g)$, the chosen $M_*$ maximizes the range of $\chi$.

Let us consider the other case 
\begin{equation}\label{critical_region2}
||\rho_0||_{L^\infty(\Om)}^{-1} \geq \al^{\frac{1}{1-r} } \bt^{\frac{1}{r-1}} (2-r)^{\frac{1}{r-1}}.
\end{equation}
By our assumption,
\begin{equation}\label{chi_star_determined2}
\lam \chi < \chi_*=\al^{\frac{2-r}{1-r}} \bt^{\frac{1}{r-1}}(2-r)^{\frac{2-r}{r-1}} (r-1).
\end{equation}
If we set $M_*:=M_f=\frac{\al(r-1)}{2-r}$, \eqref{possible.} and \eqref{possible!} are satisfied due to  \eqref{chi_star_determined2}; To be more specific, \eqref{critical_region2} is equivalent to $g'(0) \geq f(M_f)/M_f$. In this case, the maximizer of the function $\min(f,g)$ is always consistent with the maximizer of $f$.
Taking into account \eqref{maxmin_fg} again, we conclude that such a choice of $M_*$ maximizes the range of $\chi$.

Thus, we found $M_*$ fulfilling \eqref{possible.} and \eqref{possible!} in all three cases. Therefore, there is $\T_1^*>0$ such that $\theta_1(\T) \geq \theta_1(0)=||\rho_0||_{\infty}^{-1}$ whenever $\T \in (0, \T_1^*)$.  Similarly, there exists $\T_2^*>0$ such that if $\T \in (0, \T_2^*)$ then $\theta_2(\T) \geq \theta_2(0)=\eta^{-1}$.  
We denote $\T_* = \min(\T_1^*,\T_2^*)$.   This completes the proof.
\end{proof}

\vspace{3mm}
Let us introduce the following hypothesis:
$$
\textbf{(H)} 	\quad \inf_{s>0} s U''(s)>0 \quad \textrm{and} \quad 0< c_* \leq \rho_0 \in C^{0,a}(\OM).
$$
These conditions satisfy the assumptions on $U$ and $g$ (regarded as $\rho_0$) in (iii) of Proposition \ref{CS_theorem}.
Under the hypothesis \textbf{(H)}, we prove the following proposition, which ensures that the minimizers $\rho_\half$ and $\rho_\one$ are uniformly bounded. Note that we temporarily need such a technical assumption \textbf{(H)}, but it will be removed later (see \textit{Step 3} in the proof of Theorem \ref{main}).

\begin{proposition}\label{surprise}
Suppose that Assumption \ref{assumptionsUF}  and \emph{\textbf{(H)}} hold, and assume the same hypotheses as in Lemma \ref{calculus}. Let $\lam, M_*, \eta_{M_*}$, and  $\T_*$ be the numbers in Lemma \ref{calculus},
and  
let $\xi$ be the number given in Lemma \ref{L^1_bound}. 
Also, $c_0$ denotes the number given in (iii) of Proposition \ref{CS_theorem}, corresponding to $g=\rho_0$. 
Suppose that $\rho_\half$ and $\rho_\one$ are minimizers in \eqref{xrightarrow}.
We write $\eta= \eta_{_{M_*}}$ and denote
\begin{equation*}
\widehat{\T}:= \min( \frac{1}{2\al}, c_0 ||\rho_0||_{L^\infty}^{-1}, \T_*, \eta^{-1}c_0, \frac{|\Om|}{\chi \xi}),
\quad
\widetilde{\T}:= \min \Big( (2\lam \chi ||\rho_0||_{L^\infty})^{-1},(2\lam \chi \eta)^{-1} \Big).
\end{equation*}	
 If $0<\T < \min(\widehat{\T},\widetilde{\T})$, then
\begin{equation}\label{uniform_estimate3}
 \max(|| \rho_{k}^\T||_{L^\infty},   \frac{1}{2}  || \rho_{k+1/2}^\T||_{L^\infty}) \leq \max(\eta, ||\rho_0||_{L^\infty}) \quad \textrm{for all } k \in \N \cup\{0\}.
\end{equation}
\end{proposition}

\begin{proof}
By Lemma \ref{L^1_bound}, $\rho_\half$ and $\rho_\one$ exist for all $n \in \N \cup \{0\}$.
Recall the functions $\theta_i$, $i=1,2$, in Lemma \ref{calculus}.

Due to the assumption $\textbf{(H)}$, the hypotheses of (iii) in Proposition \ref{CS_theorem} are satisfied. Then we have
$$
||\rho_{1/2}||_{\infty} \leq \frac{||\rho_0||}{1-\lam \T \chi ||\rho_0||_\infty}.
$$
Moreover, by Lemma \ref{essential_first_estimate}, we have
\begin{align}\label{half_half_then_one}
||\rho_1||_\infty 
&\leq
\max(
   \frac{||\rho_{1/2}||_{\infty}}{1+\T M}, \eta)
\leq 
\max(\frac{||\rho_0||_{\infty}}{(1+\T M)(1- \lam \T \chi ||\rho_0||_{\infty})} , \eta)
\nn \\
&
= \max(\frac{1}{\theta_1(\T)} ,\eta)
\leq \max(\frac{1}{\theta_1(0)}, \eta)=
\max(||\rho_0||_\infty, \eta).
\end{align}

From \eqref{half_half_then_one}, we know $\T ||\rho_1|| < c_0$. Hence, thanks to Lemma \ref{preserved_supports}, we can apply 
(iii) of Proposition \ref{CS_theorem} again, that is,
$$
||\rho_{3/2}||_{\infty} \leq \frac{||\rho_1||_{\infty}}{1- \lam \T \chi ||\rho_1||_{\infty}},
$$
where $\rho_{3/2}$ is any minimizer determined by \eqref{argmin}.
Using Lemma \ref{essential_first_estimate} again, as in the proof of \eqref{half_half_then_one}, we obtain
$$
||\rho_2||_{\infty} \leq \max(||\rho_0||_\infty , \eta).
$$
Therefore, the following holds by induction:
\begin{equation}\label{needed_uniform_estimate}
|| \rho_{k}^\T||_{L^\infty}   \leq \max(\eta, ||\rho_0||_{L^\infty}) \quad \textrm{for all } k \in \N \cup\{0\}.
\end{equation}

It remains to show
\begin{equation*}
|| \rho_{k+1/2}^\T||_{L^\infty}   \leq 2 \max(\eta, ||\rho_0||_{L^\infty}) \quad \textrm{for all } k \in \N \cup\{0\},
\end{equation*}
which is a direct consequence of \eqref{CS_estimate} and \eqref{needed_uniform_estimate}. Indeed, since $\T< \widehat{\T}$ and \eqref{needed_uniform_estimate} holds, we have $\T ||\rho_n||_{\infty} \leq c_0$. Using (iii) in Proposition \ref{CS_theorem} and \eqref{needed_uniform_estimate}, we see that
$$
||\rho_{n+1/2}||_{\infty}^{-1} \geq ||\rho_n||_{\infty}^{-1} - \lam \T \chi \geq \min (\eta^{-1} , ||\rho_0||_{\infty}^{-1}) - \lam \T \chi \geq \frac{1}{2} \min(\eta^{-1}, ||\rho_0||_{\infty}^{-1}),
$$
where we used the assumption $\T < \widetilde{\T}$ in the last inequality.
Thus, we obtain the desired result \eqref{uniform_estimate3}.
\end{proof}

\vspace{3mm}

In the following lemma, we record the $L^\infty$ distance estimate between
$\rho^\tau_{n+1/2}$ and $\rho^\tau_{n+1}$, namely the two minimizers before and
after the Fisher-Rao step, together with the energy gap estimates and the
$\WFR$ distance estimates used later.

\begin{lemma}
\label{stepping_stones}
Suppose that the assumptions in Proposition \ref{surprise} hold, and we denote $C_1:=2\max( \eta_{M^*}, ||\rho_0||_{L^\infty(\Om)})$.
There exists $\T_{**}>0$ such that
if $\T < \T_{**}$, then 
\begin{itemize}
\item[\emph{(i)}]
for all $t >0$
\begin{equation}\label{two_solutions_gap}
||  \tilde{\rho}^\T(t) - \rho^\T(t)  ||_{L^{\infty}(\Om)} \leq C \T,
\end{equation}
where $\tilde{\rho}^\T(t)$ and $\rho^\T(t)$ are defined in \eqref{two_curves_definition}.

\item[\emph{(ii)}] 
for all $n\in \N \cup \{0\}$ 
\begin{equation}\label{halfstep_energy_gap}
\cE_{1}(\rho_{n+1}^\T)  - \cE_{1}(\rho_{n+1/2}^\T)
\leq  C \T,
\end{equation}
where $C=C(C_1, \al, \Om, \chi, K_3)>0$ with the number $K_3$ in \eqref{elliptic_estimate}.

\item[\emph{(iii)}]
for all $n\in \N \cup \{0\}$
\begin{equation}\label{halfstep_energy_gap_2}
\cE_2(\rho_\half) - \cE_2 (\rho_\one) \leq 
C \T,
\end{equation}
where $C=C(C_1, \Om)$ is a positive constant.

\item[\emph{(iv)}]
Let $0<T<\infty.$ For all $0 \leq s \leq t \leq T$
\begin{equation}\label{WFR_continuity}
    \emph{WFR}^2(\tilde{\rho}^\T (s) , \tilde{\rho}^\T(t)) \leq C ( t-s+\T), \quad \emph{WFR}^2(\rho^\T (s) , \rho^\T(t)) \leq C ( t-s+\T),
\end{equation}
where $C=C(C_1, \al, \Om, \chi, K_3, T)>0$. 

\end{itemize}
\end{lemma}

\begin{proof} 
(i) Fix $t \in ( n\T ,(n+1)\T]$. Using (ii) of Lemma \ref{preserved_supports} and \eqref{Gallouet_pointwise}, we obtain
\begin{equation}\label{FR_optimality}
\sqrt{\rho_\one} - \sqrt{\rho_\half} = -\frac{\T}{2} \sqrt{\rho_\one}  F'(\rho_\one) \quad  \cL^d \textrm{-a.e.}
\end{equation}
By Proposition \ref{surprise} and the continuity of $F'$, $F'(\rho_{n+1}^\T)$ is uniformly bounded in $n$. Multiplying \eqref{FR_optimality} by 
$\sqrt{\rho_\one} + \sqrt{\rho_\half}$, we deduce that 
\begin{equation}\label{half_gap_L_infty}
||\rho_\one - \rho_\half||_{L^\infty(\Om)} \leq \T C_1 \sup_{z \in [0,C_1]}  |F'(z)|.
\end{equation}

(ii) Let us first show that 
\begin{equation}\label{U_half}
\int U(\rho_\one) - \int U(\rho_\half) \leq  \big( 2\al  C_1 U'(2C_1) - 3k_0   \inf_{s>0} s U'(s) \big)|\Om| \T,
\end{equation}
where
$$
k_0:=\sup_{z \in [0,C_1]} |F'(z)|.
$$
To prove this, we note from \eqref{FR_optimality} that
$$
\sqrt{\rho_\one} - \sqrt{\rho_\half} \geq -\frac{k_0}{2}\T \sqrt{\rho_\one}   \quad  \cL^d \textrm{-a.e.}
$$
This implies that there exists $\T_{**}>0$ such that if $\T<\T_{**}$ then
\begin{equation}\label{easy_estimate}
    \rho_\one \geq (1-2k_0 \T)\rho_\half \quad \textrm{$\cL^d$-a.e.}
\end{equation}
On the other hand, using \eqref{FR_optimality} again,  similar to \eqref{Gallouet_square_a.e.}, we  have
\begin{equation}\label{easy_estimate2}
    \rho_\half - \rho_\one \geq \T \rho_\one F'(\rho_\one)  
    \geq - \al \T \rho_\one \quad \textrm{$\cL^d$-a.e.}
\end{equation}

Since we assumed \textbf{(H)}, $U$ is increasing on $(r, \infty)$ and is decreasing on $[0, r)$ (of course, $r$ can be $0$, in which case the proof is easier).
From \eqref{easy_estimate2}, there exists $\T_{**}>0$ (not relabeled) such that for all $n$, $\rho_\one \leq (1+2\al \T) \rho_\half$  whenever $\T< \T_{**}$. Then
\begin{align}
\int_{ \{\rho_\one \geq r \}} U(\rho_\one) 
&\leq \int_{\{\rho_\one \geq r\}} U((1+2\al \T)\rho_\half)
\nn \\
&  \leq  \int_{ \{\rho_\one \geq r \}} U(\rho_\half) + 2\al \T \rho_\half  \sup_{z \in [0,2\al]}U' \big((1+ z \T) \rho_\half \big)  \nn \\
&    \leq  \int_{ \{\rho_\one \geq r \}} U(\rho_\half) + 2\al \T C_1 U'(2C_1),     \label{U_increasing_part}
\end{align}
where we used Proposition \ref{surprise} and the fact that $U'$ is continuous. To estimate $\int_{ \{\rho_\one < r \}} U(\rho_\one)$, we note that $\inf_{s>0} s U'(s) \leq 0$ due to the definition of $r$ (if $r=0$, the following estimate is not needed). 
Using \eqref{easy_estimate} and assuming $6 k_0\T<1$, still denoted by $\T< \T_{**}$, we have
\begin{align}
\int_{ \{\rho_\one < r \}} U(\rho_\one)  
&\leq \int_{ \{\rho_\one < r \}} U((1-2k_0 \T)\rho_\half) 
\nn \\
&  \leq  \int_{ \{\rho_\one < r \}} U(\rho_\half) -2k_0 \T \rho_\half U'((1-2k_0 \T) \rho_\half)
  \nn \\
&\leq \int_{ \{\rho_\one < r \}} U(\rho_\half) - \frac{2k_0 \T}{1-2k_0 \T} \inf_{s \in [0,C_1] } s U'(s)
  \nn \\
  &\leq \int_{ \{\rho_\one < r \}} U(\rho_\half) - 3k_0 \T  \inf_{s \in [0,C_1]} s U'(s).
  \nn 
\end{align}
Then we obtain \eqref{U_half}.

We next show that
$$
\bigg| \int_\Om |\grad c[\rho_\one]|^2 +\big(c[\rho_\one])^2 dx
 - \int_\Om |\grad c[\rho_\half]|^2 +\big(c[\rho_\half])^2 dx \bigg|
$$
\begin{equation}\label{half_gap_goal}
 \leq 4K_3^2   C_1^2 \sup_{z \in [0,C_1]} |F'(z)|    |\Om|\T,
\end{equation}
from which the desired result \eqref{halfstep_energy_gap} follows.
Let us denote $\bar{\rho}_n := \rho_\one - \rho_\half$ and $\bar{c}_n:= c[\rho_\one] - c[\rho_\half]$. Clearly, by the linearity, $( \bar{\rho}_n, \bar{c}_n )$  solves  \eqref{elliptic}. Sobolev embedding theorems and \eqref{elliptic_estimate} imply that for large $p$,
$$
||\bar{c}_n||_{C^{1,\gam}(\OM)} \leq K_2 ||\bar{\rho}_n||_{L^p(\Om)} \leq K_3 || \bar{\rho}_n||_{L^\infty(\Om)} \leq K_3 C_1 \sup_{z \in [0,C_1]} |F'(z)|  \T,
$$
where we have used \eqref{half_gap_L_infty}.
Thus, we get
$$
\bigg|  \int_\Om \big(c[\rho_\one])^2 - \big(c[\rho_\half])^2 \bigg|
\leq \int_\Om |\bar{c}_n |\big( c[\rho_\one]+ c[\rho_\half]  \big) \leq 2K_3^2   C_1^2 \sup_{z \in [0,C_1]} |F'(z)| |\Om|  \T.
$$
Estimating similarly $\int_\Om (|\grad c[\rho_\one]|^2 - |\grad c[\rho_\half]|^2$),
we obtain \eqref{half_gap_goal}.

(iii)
Since $F''>0$, there exists $r' \geq 0$ such that $F$ is decreasing on $[0,r')$ and increasing on $(r',\infty)$ (If $r' = 0$, the argument becomes even simpler).

Similar to \eqref{U_increasing_part}, we can observe that
\begin{equation}\label{E2_half_gap1}
\int_{ \{\rho_\half \geq r' \}} F(\rho_\half) \leq 
\int_{ \{ \rho_\half \geq r' \}} F(\rho_\one) +4k_0 \T C_1 F'(2C_1).
\end{equation}
Using that $\rho_\half \geq (1-\al \T)\rho_\one$ in \eqref{implicit_form_positivity} and $F'\geq -\al$, we have
\begin{align}
\int_{ \{\rho_\half < r' \}} F(\rho_\half) &\leq 
\int_{ \{ \rho_\half < r' \}} F((1-\al \T)\rho_\one) 
\nn \\
&\leq \int_{ \{\rho_\half < r' \}} F(\rho_\one)  -\al \T \rho_\one F'((1-\al \T)\rho_\one)
\nn \\
&\leq \int_{ \{\rho_\half < r' \}} F(\rho_\one)  +\al^2 \T C_1.
\label{E2_half_gap2}
\end{align}
The desired result is obtained by combining \eqref{E2_half_gap1} and \eqref{E2_half_gap2}.

(iv) Let $0 \leq s \leq t \leq T$. Obviously, $s \in \big( N_1 \T, (N_1+1)\T \big]$ and $t \in \big( N_2 \T, (N_2+1)\T \big]$ for some $N_1,N_2 \leq \lfloor \frac{T}{\T}  \rfloor$.
$$
\WFR(\tilde{\rho}^\T (s), \tilde{\rho}^\T (t))=\WFR(\rho_{N_1+1/2}^\T, \rho_{N_2+1/2}^\T) 
$$
\begin{equation}\label{WFR_conti_1}
\leq \sum_{k=N_1}^{N_2-1} \WFR(\rho_{k+1/2}^\T, \rho_{k+3/2}^\T) \leq (\sum_{k=N_1}^{N_2-1}  \T)^{1/2} \bigg(\sum_{k=N_1}^{N_2-1}   \frac{\WFR^2(\rho_{k+1/2}^\T, \rho_{k+3/2}^\T)}{\T} \bigg)^{1/2}. 
\end{equation}
From \eqref{pytagorean}, \eqref{argmin}, and \eqref{argmin2}, we see that
$$
\WFR^2(\rho_{k+1/2}^\T, \rho_{k+3/2}^\T) \leq
2 \Big(\FR^2 (\rho_{k+1/2}^\T, \rho_{k+1}^\T) + W_2^2(\rho_{k+1}^\T,\rho_{k+3/2}^\T) \Big),
$$
\begin{equation}\label{from_1st_scheme}
\frac{W_2^2(\rho_{k+1}^\T,\rho_{k+3/2}^\T)}{\T} \leq \cE_1( \rho_{k+1}^\T) - \cE_1(\rho_{k+3/2}^\T),
\end{equation}
$$
\frac{\FR^2(\rho_{k+1/2}^\T,\rho_{k+1}^\T)}{2\T} \leq \cE_2( \rho_{k+1/2}^\T) - \cE_2(\rho_{k+1}^\T). 
$$

\noindent
Therefore,
\begin{align}
&
\sum_{k=N_1}^{N_2-1}   \frac{\WFR^2(\rho_{k+1/2}^\T, \rho_{k+3/2}^\T)}{2\T}
\leq
\sum_{k=N_1}^{N_2-1} 
\frac {\FR^2 (\rho_{k+1/2}^\T, \rho_{k+1}^\T) + W_2^2(\rho_{k+1}^\T,\rho_{k+3/2}^\T) }{\T}
\nn \\
&\leq 
\sum_{k=N_1}^{N_2-1}
 \cE_1( \rho_{k+1}^\T) - \cE_1(\rho_{k+3/2}^\T)  +2 \big( \cE_2( \rho_{k+1/2}^\T) - \cE_2(\rho_{k+1}^\T) \big)
\nn \\
&=
 \cE_1( \rho_{N_1+1}^\T) + \Big(- \cE_1(\rho_{N_1+3/2}^\T) + \cE_1( \rho_{N_1+2}^\T) \Big) 
\nn \\
& +... + \Big(-\cE_1( \rho_{N_2-1/2}^\T)  + \cE_1(\rho_{N_2}^\T) \Big) -\cE_1(\rho_{N_2+1/2}^\T) 
+
2 \sum_{k=N_1}^{N_2-1} \big( \cE_2( \rho_{k+1/2}^\T) - \cE_2(\rho_{k+1}^\T)\big).
 \nn
\end{align}
Thanks to \eqref{halfstep_energy_gap} and \eqref{halfstep_energy_gap_2}, we obtain
\begin{align}
&\sum_{k=N_1}^{N_2-1}   \frac{\WFR^2(\rho_{k+1/2}^\T, \rho_{k+3/2}^\T)}{2\T}
\nn \\
&
\leq   \sup_{\rho \leq C_1} \cE_1(\rho) + C(N_2-N_1-1)\T  - \inf_{\rho \leq C_1} \cE_1(\rho)    + C (N_2-N_1)\T
 \nn \\
&
\leq   \sup_{\rho \leq C_1} \cE_1(\rho) + CT  - \inf_{\rho \leq C_1} \cE_1(\rho)
 + C T
 \label{WFR_conti_2}
\end{align}
for some constant $C>0$. Here, the supremum and infimum appearing in \eqref{WFR_conti_2} are finite by 
 Proposition \ref{surprise}.

Combining \eqref{WFR_conti_1} and \eqref{WFR_conti_2}, we get the first inequality in
\eqref{WFR_continuity}. The second one follows in a similar way.
\end{proof}

\vspace{3mm}
Now, we obtain an approximate weak formulation of \eqref{KSL} by using the curves $\tilde{\rho}^\T(t), \rho^\T(t)$.

\begin{lemma}\label{combining_ideas}
Let $T>0$.
Suppose that the assumptions in Proposition \ref{surprise} hold and that $\T< \T_{**}$ as stated in Lemma \ref{stepping_stones}. 
Let $\Psi(\rho):= \rho U'(\rho) -U(\rho)$. For all $\phi \in C^\infty(\OM)$,
we have
$$
\int_\Om \big(\rho^\T(t_2) - \rho^\T(t_1)  \big) \phi
+\mathscr{R}(\T)
=-\int_{t_1}^{t_2} \int_\Om \frac{\sqrt{\rho^\T}(\sqrt{\rho^\T} + \sqrt{\tilde{\rho}^\T } )}{2} F'(\rho^\T) \phi
$$
\begin{equation}\label{approximate_weak_solution}
-\int_{t_1}^{t_2} \int_\Om \grad \phi \cdot \grad \Psi(\tilde{\rho}^\T)
+\int_{t_1}^{t_2} \int_\Om \chi \tilde{\rho}^\T \grad c[\tilde{\rho}^\T] \cdot \grad \phi
\qquad \forall 0 \leq t_1 < t_2 \leq T,
\end{equation}
where the remainder term $\mathscr{R}(\T)$ tends to $0$ as $\T \downarrow 0$.
\end{lemma}

\begin{proof}
First, we recall the Euler-Lagrange equation \eqref{Santambrogio!}.
Using Proposition \ref{surprise}, $\rho_\half$ is always smaller than $\overline{B}:=\frac{1}{\chi \T}$ whenever $\T$ is small enough, which means $p \equiv 0$ in \eqref{Santambrogio!}. Hence, we get
\begin{equation}\label{load_EL1}
U'(\rho_\half) -\chi c[\rho_\half] + \frac{\vphi}{\T} = l \quad \textrm{for all $x \in \OM$}.
\end{equation}
Since all functions in \eqref{load_EL1} are Lipschitz, 
\begin{equation}\label{OC1}
\grad (U'(\rho_\half )) - \chi \grad c[\rho_\half] + \frac{id-T_n}{\T} =0 \quad a.e. ~x \in \Om,
\end{equation}
where $T_n$ is an optimal transport map from $\rho_\half$ to $\rho_n^\T$. Since $\rho_n^\T=(T_n)_\# \rho_\half$, for all $\phi \in C^\infty(\OM)$ we have
$$
\int \phi  \rho_\half - \int \phi  \rho_n^\T = \int \phi \rho_\half  - \int \phi \circ T_n \rho_\half
$$
$$
=-\int  ( \phi(T_n(x)) - \phi(x) ) \rho_\half = -\T \int \grad \phi(x) \cdot \frac{T_n(x)-x}{\T} \rho_\half + E_n(\T)
$$
where $|E_n(\T)| \leq \sup |D^2 \phi|W_2^2(\rho_n^\T, \rho_\half)$. Then
$$
\int \phi  \rho_\half - \int \phi  \rho_n^\T 
= -\T \int \grad \phi(x) \cdot (  \grad U'(\rho_\half)- \chi \grad c )\rho_\half+ E_n(\T)
$$
\begin{equation}\label{making_WS}
= -\T \int \grad \phi(x)  \cdot \grad \Psi(\rho_\half)
+ \T \int  \chi \rho_\half \grad c \cdot \grad \phi 
+E_n(\T),
\end{equation}
where in the last equality we have used that $\rho \grad U'(\rho) = \grad (\rho U'(\rho) - U(\rho)) = \grad \Psi(\rho)$.

Finally, 
we fix $0 \leq t_1 < t_2 \leq T$.
We will first show that \eqref{approximate_weak_solution} holds for the case $t_1>0$. 
Let $t_1>0$, and then obviously we have $t_1 \in \big( (n-1)\T, n\T \big]$ and $t_2 \in \big( (m-1)\T, m\T \big]$ for some $m,n \in \N$.
Since
$$
\int_\Om \phi \big(\rho^\T(t_2) - \rho^\T(t_1) \big) 
=\int_\Om \phi \big(\rho_{m}^\T - \rho_ {n}^\T  \big)
$$
$$
=\sum_{k=n}^{m-1} \int_\Om \phi  \big(\rho_{k+1}^\T - \rho_ {k}^\T  \big)
=\sum_{k=n}^{m-1} \int_\Om \phi  \big(\rho_{k+1}^\T - \rho_{k+1/2}^\T) 
+\phi \big(\rho_{k+1/2}^\T 
- \rho_ {k}^\T  \big),
$$
by using Proposition \ref{surprise}, \eqref{Gallouet!}, and \eqref{making_WS}, we obtain
$$
\int_\Om \phi  \rho^\T(t_2) - \int_\Om \phi  \rho^\T(t_1)
=-\int_{n \T}^{m \T} \int_\Om  \frac{\sqrt{\rho^\T}(\sqrt{\rho^\T} + \sqrt{\tilde{\rho}^\T } )}{2}   F'(\rho^\T)     \phi    
$$
\begin{equation}\label{semi_approximate_weak_solution}
-\int_{n \T}^{m \T} \int_\Om \grad \phi
\cdot \grad \Psi(\tilde{\rho}^\T)
+\int_{n \T}^{m \T} \int_\Om \chi \tilde{\rho}^\T \grad c[\tilde{\rho}^\T] \cdot \grad \phi
+ E(\T)
\qquad \forall  \phi \in C^\infty(\OM),
\end{equation}
where  $|E(\T) | \leq \sum_{k=n}^{m-1} \sup |D^2 \phi|W_2^2(\rho_k^\T, \rho_{k+1/2}^\T)$.  
Using \eqref{halfstep_energy_gap} and \eqref{from_1st_scheme}
(as in \eqref{WFR_conti_2}) we have
$$
|E(\T)| \leq \T ||D^2 \phi||_{L^\infty} \big[CT + \sup_{\rho \leq C_1} \cE_1(\rho) - \inf_{\rho \leq C_1} \cE_1(\rho) \big] 
$$ 
with the constant $C_1$ given in Lemma \ref{stepping_stones}.
If $t_1 = 0$, we set $\rho^\tau(0) = \rho_0$ and take $n = 0$ in the above argument. 
The rest of the proof remains valid.

Let $\mathscr{R}(\T)$ be the subtraction of the right hand side in \eqref{semi_approximate_weak_solution} from the right hand side in \eqref{approximate_weak_solution}.
Thanks to Proposition \ref{surprise}, since
$\sqrt{\rho^\T}(\sqrt{\rho^\T} + \sqrt{\tilde{\rho}^\T } ) F'(\rho^\T), ~
\tilde{\rho}^\T U'(\tilde{\rho}^\T) - U(\tilde{\rho}^\T)$, and  $\tilde{\rho}^\T \grad c[\tilde{\rho}^\T]$ belong to  $L^1((0,T) \times \Om))$,  we see that
$\mathscr{R}(\T) \ra 0$ as $\T \ra 0$.
\end{proof}

\vspace{5mm}
\section{Proof of Theorem \ref{main}}\label{section_proofs}
We prove Theorem \ref{main} sequentially. More precisely, we will first prove the theorem with the additional assumption \textbf{(H)}. Finally, we will show that the assumption \textbf{(H)} can be removed, and finally, Theorem \ref{main} is proved.

\textit{Proof of Theorem \ref{main}}:
Fix $T>0$. 
Let $\T$ be small enough to satisfy the assumptions of Lemma \ref{combining_ideas}. Let $(\Tk)$ be a vanishing sequence that satisfies such a smallness. It is important to note that we are assuming \textbf{(H)} in \textit{Step 1} and \textit{Step 2}.

\vspace{3mm}
\textit{Step 1. $\frac{1}{2}$-Hölder continuity w.r.t. \emph{$\WFR$} metric} :
To prove this step, we follow the main lines of \cite[Corollary 4.1]{GM17} (in our case, thanks to Proposition \ref{surprise}, we do not need to estimate the masses by the $\WFR$ distance as before; moreover, our limit measure belongs to $L^\infty(\Om)$). Clearly, by Proposition \ref{surprise}, we have
\begin{equation}\label{alltime_boundedness}
||\rho^{\Tk}(t)||_{L^\infty(\Om)} \leq C_1, \qquad ||\tilde{\rho}^{\Tk} (t)||_{L^\infty(\Om)} \leq C_1     \qquad \forall t \in [0,T],
\end{equation}
where $C_1$ is the constant defined in Lemma \ref{stepping_stones}. 
 
Using \eqref{alltime_boundedness}, we know that $ \rho^{\Tk}(t), ~\tilde{\rho}^{\Tk} (t)$ belong to the set $\{ \si \in \cM^+(\OM) : |\si| \leq |\Om| C_1 \}$, which is relatively compact in $\cM^+(\OM)$. Since \eqref{WFR_continuity} holds, thanks to a refined version of the Ascoli-Arzel\`{a} theorem \cite[Proposition 3.3.1]{AGS08}, we extract a subsequence $(\Tk)$ (not relabeled) such that for all $t \in [0,T]$
$$
    \rho^{\Tk}(t)  \wsra \nu(t) \quad \textrm{in $C_c(\OM)^*=C(\OM)^*$ as $k \ra \infty$}, 
$$
which means 
$$
\int_{\OM} \rho^{\Tk}(t) \phi dx \ra \int_{\OM} \phi d\nu(t) \qquad  \forall \phi \in C(\OM).
$$
Thanks to \eqref{alltime_boundedness}, we conclude that there exists a curve 
$\rho:[0,T] \ra L^\infty(\Om)$
such that 
$$
\nu(t)=\rho(t) \cL^d|_{\OM}  \qquad \forall t \in [0,T].
$$

Repeating this process to $\tilde{\rho}^{\Tk}$, up to a subsequence, we have
$$
\int_{\OM} \tilde{\rho}^{\Tk}(t) \phi dx \ra \int_{\OM} \rho^*(t) \phi dx \qquad  \forall \phi \in C(\OM),
$$
for a curve $\rho^*:[0,T] \ra L^\infty(\Om)$.

On the other hand, from \eqref{WFR_leq_FR} and (iii) of Lemma \ref{stepping_stones}, we note $\WFR(\rho^{\Tk} (t), \tilde{\rho}^{\Tk} (t) )\leq \sqrt{2C} \Tk.$ As mentioned in Remark \ref{WFR_metric_remark},
since $\WFR$ is lower semicontinuous with respect to weak* convergence, we obtain $\WFR(\rho(t), \rho^*(t)) \leq \liminf_{k \ra \infty} \WFR(\rho^{\Tk} (t), \tilde{\rho}^{\Tk} (t) )\leq 0$, which gives $\rho= \rho^*$.
Therefore, for all $t \in [0,T]$ and all $\phi \in C(\OM)$,
\begin{equation}\label{same_limit}
\int_{\OM} \rho^{\Tk}(t) \phi dx, ~ \int_{\OM} \tilde{\rho}^{\Tk}(t) \phi dx 
\ra \int_{\OM} \rho(t)\phi  dx.
\end{equation}
Recalling \eqref{WFR_continuity} and using again
the lower semicontinuity of $\WFR$, we get
\begin{equation}\label{real_WFR_continuity}
\WFR(\rho (t), \rho(s)) \leq C |t-s|^{1/2} \qquad \forall s,t \in [0,T].
\end{equation}
This means that $\rho \in C^{1/2}([0,T]; \cM_{\WFR}^+(\Om))$.

\vspace{3mm}
\textit{Step 2. Existence of a weak solution \emph{\&} \eqref{The_estimate}} :
We wish to pass to the limit in \eqref{approximate_weak_solution} as $\T \downarrow 0$.
The limit on the left hand side follows directly from \eqref{same_limit}, i.e.,
\begin{equation}\label{LHS_conv}
 \int_\Om (\rho^{\Tk}(t_2) -\rho^{\Tk}(t_1)) \phi \ra   \int_\Om (\rho(t_2) -\rho(t_1)) \phi   \qquad   \forall \phi \in C^\infty(\OM).
\end{equation}

From \eqref{OC1}, we know that
\begin{equation}\label{take_grad_OC1}
\grad (U'(\rho_\halfk )) - \chi \grad c[\rho_\halfk] = \frac{T_n-id}{\Tk}  \quad a.e.
\end{equation}
where $T_n$ is an optimal transport map from $\rho_\halfk$ to $\rho_\placek$. Combining \eqref{take_grad_OC1} with the following inequality
$$
\frac{W_2^2(\rho_\placek, \rho_\halfk) }{\Tk} 
\leq \cE_{1}(\rho_\placek) - \cE_{1}(\rho_\halfk),
$$
we obtain
\begin{equation}\label{grad_control1}
\Tk  \int_\Om | \grad (U'(\rho_\halfk )) - \chi \grad c[\rho_\halfk] |^2 \rho_\half dx 
\leq \cE_{1}(\rho_\placek) - \cE_{1}(\rho_\halfk).
\end{equation}
Using \eqref{halfstep_energy_gap} and \eqref{grad_control1}, we have
\begin{equation}\label{previous_sum_up}
\Tk \int_\Om | \grad (U'(\rho_\halfk )) - \chi \grad c[\rho_\halfk] |^2 \rho_\halfk dx 
\leq \cE_{1}(\rho_\placek) - \cE_{1}(\rho_\onek) + C \Tk
\end{equation}
for some $C=C(C_1, \al, \Om, \chi, K_3)>0$.
Adding up \eqref{previous_sum_up} from $n=0$ to $n=N_k$ (with $N_k:=\lfloor \frac{T}{\Tk} \rfloor)$,  we see that
$$
    \int_0^T  || \grad U'( \tilde{\rho}^{\Tk}(t) ) - \chi \grad c[ \tilde{\rho}^{\Tk}(t)] ||_{L^2(\tilde{\rho}^{\Tk}(t)dx)}^2 dt 
\leq \cE_1(\rho_0) - \inf_{\rho \leq C_1} \cE_1(\rho)  + C(N_k+1) \Tk 
$$
\begin{equation}\label{a_priori_estimate}
\leq \cE_1(\rho_0) - \inf_{\rho \leq C_1} \cE_1(\rho) + C(T+1),
\end{equation}
where we used Proposition \ref{surprise} to get the uniform boundedness of $\cE_{1}$.

Applying Young's inequality to \eqref{a_priori_estimate}, we obtain
$$
\int_{\Om \times (0,T)} 
\tilde{\rho}^{\Tk} | \grad U'(\tilde{\rho}^{\Tk})|^2 
 \leq
\cE_1(\rho_0) - \inf_{\rho \leq C_1} \cE_1(\rho) + C(T+1)
 +\ep \int_{\Om \times (0,T)} \tilde{\rho}^{\Tk} | \grad U'(\tilde{\rho}^{\Tk})|^2  
$$
\begin{equation}\label{Young's}
+ (\frac{\chi^2}{\ep} - \chi^2) \int_{\Om \times (0,T)} |\grad c[ \tilde{\rho}^{\Tk}]|^2 \tilde{\rho}^{\Tk}.
\end{equation}
On the other hand, using \eqref{alltime_boundedness} and \eqref{elliptic_estimate}, similar to \eqref{half_gap_goal}, we get
\begin{equation}\label{grad_c}
\int_{\Om \times (0,T)} |\grad c [ \tilde{\rho}^{\Tk}]|^2 \tilde{\rho}^{\Tk} \leq K_3^2 C_1^3 \big|\Om \times (0,T) \big|.
\end{equation}
Combining \eqref{Young's} and \eqref{grad_c} and choosing $\ep=1/2$, we see that
\begin{equation}\label{W1,1}
  \int_{\Om \times (0,T)} \tilde{\rho}^{\Tk} | \grad U'(\tilde{\rho}^{\Tk})|^2
  \leq C=C(C_1, \al, \Om, \chi, K_3, T),
\end{equation}
which gives
\begin{equation}\label{H^1}
  \int_{\Om \times (0,T)} |\grad \Psi(\tilde{\rho}^{\Tk}) |^2=  \int_{\Om \times (0,T)} |\tilde{\rho}^{\Tk} \grad U'(\tilde{\rho}^{\Tk})|^2
  \leq C=C(C_1, \al, \Om, \chi, K_3, T).
\end{equation}

Now, in order to prove the following claim, we use a variant of the Aubin-Lions lemma \cite{RS03} introduced in Proposition \ref{ALRS}.

\textbf{Claim:}
There exists a subsequence $(\T_k)$ (not relabeled) such that for a.e. $t \in (0,T)$ and for all $ p \in (1, \infty)$,
\begin{equation}\label{strongLp}
    \tilde{\rho}^{\Tk}(t) \ra \rho^{**}(t) \quad \textrm{strongly in} \quad L^p(\Om),
\end{equation}
\begin{equation}\label{StrongLp}
    \tilde{\rho}^{\Tk} \ra \rho^{**} \quad \textrm{strongly in} \quad L^p(\Om \times (0,T)).
\end{equation}

\textit{Proof of \emph{\textbf{Claim}}}:
Let $X:= L^p(\Om)$ with $p>1$.  We write $L_{C_1}^\infty(\Om):= \{ \rho \in L^\infty(\Om) \cap \cM^+(\Om): \rho \leq C_1 ~a.e. \}$. Then we set
\begin{equation}\label{AubinLions_functional}
\mathscr{F}(\rho):=
  \begin{cases}
\displaystyle{\int_\Om |\grad \Psi(\rho)|^2 } & \text{if $\grad \Psi(\rho) \in L^2(\Om)$ and   $\rho \in L_{C_1}^\infty(\Om)$},             
           \\    \infty & \text{else,}
  \end{cases}
\end{equation}
where $\Psi(\rho) := \rho U'(\rho) -U(\rho)$, and define the pseudo distance $g$ on $L^p(\Om) \times L^p(\Om)$ as follows:
$$
g(u,\tilde{u}):=
  \begin{cases}
\WFR(u,\tilde{u}) & \text{if $u,\tilde{u} \in \cM^+(\Om)$},             
           \\    \infty & \text{else.}
  \end{cases}
$$
We first show that
 sublevels of $\mathscr{F}$ are relatively compact in $X$. 
Given $L>0$, let $\mathscr{F}_{\leq L} := \{\rho \in L^p(\Om):  \mathscr{F} \leq L  \}$.  If $(\rho_n)$ is any sequence in $\mathscr{F}_{\leq L}$, then, 
$$
\int_\Om   |\grad \Psi(\rho_n)|^2 \leq L,
$$
which implies by Rellich's theorem, up to a subsequence, 
$\Psi(\rho_n) \ra \mu$ in $L^2(\Om)$ strongly. Therefore, up to a subsequence,
$$
\rho_n \ra \Psi^{-1}(\mu) \quad a.e.
$$
where we have used that $\Psi'(s)>0$ for all $s>0$.
Since $\rho_n \leq C_1$, we see that
$$
\rho_n \ra \Psi^{-1}(\mu) \quad \textrm{in} \quad L^p(\Om),
$$
which gives us the desired result.

We now show that $\mathscr{F}$ is lower semicontinuous with respect to  $L^p(\Om)$ convergence, that is,
\begin{equation}\label{F_LSC}
\liminf_{n \ra \infty} \mathscr{F}(\rho_n) \geq \mathscr{F}(\rho)
\end{equation}
whenever $\rho_n \ra \rho$ in $L^p(\Om)$. Suppose that $\rho_n \ra \rho$ in $L^p(\Om)$. To show \eqref{F_LSC}, without loss of generality, we can assume that $\liminf_{n \ra \infty} \mathscr{F}(\rho_n) < \infty$, and hence there exists a subsequence $(n')$  such that
$$
\liminf_{n \ra \infty} \mathscr{F}(\rho_n) = \lim_{n' \ra \infty} \mathscr{F}(\rho_{n'}) \quad \textrm{and} \quad \sup_{n'} \mathscr{F}(\rho_{n'}) < \infty.
$$
Due to the strong convergence of $(\rho_n)$ in $L^p(\Om)$, by passing to a subsequence, $\rho_{n'} \ra \rho$ a.e.
Using weak compactness in $L^2(\Om)$, up to a subsequence, we obtain
$$
\grad \Psi (\rho_{n'}) \wra \grad \Psi (\rho) \quad \textrm{in} \quad L^2(\Om),
$$
where we used that $\rho_{n'} \ra \rho$ a.e.

The elementary inequality $|x|^2-|y|^2 \geq 2(x-y)\cdot y$ yields
\begin{equation}\label{weak_L2_LSC}
\mathscr{F}(\rho_{n'}) - \mathscr{F}(\rho) 
\geq 2\int_\Om \grad \Psi(\rho) \cdot \big(  \grad \Psi(\rho_{n'}) - \grad \Psi(\rho) \big).
\end{equation}
Letting $n' \ra \infty$, we get \eqref{F_LSC}. 

Thus, the assumption \textit{(a)} of Proposition \ref{ALRS} is satisfied. By the definition of $\mathscr{F}$ and $g$, the assumption \textit{(b)} is also clearly fulfilled. Set $\mathscr{U}:= \{ \tilde{\rho}^{\Tk} \}$, where $(\Tk)$ is the sequence in  \eqref{W1,1}.
Then \eqref{time_H^1} holds due to \eqref{H^1}.
In addition, \eqref{AL_limsup}  follows from \eqref{WFR_continuity}.
\qed\\

We go back to the proof.
Due to \eqref{same_limit},  one can check that 
\begin{equation}\label{StrongLimit}
\rho^{**}=\rho.
\end{equation}
Combining \eqref{StrongLp} and \eqref{StrongLimit}, 
 $\tilde{\rho}^{\Tk}$ converges to $\rho$ in $L^p(\Om \times (0,T))$.
Moreover, by \eqref{two_solutions_gap},
\[
||\rho^{\tau_k}-\tilde\rho^{\tau_k}||_{L^p(\Omega\times(0,T))}
\le C |\Omega|^{1/p}T^{1/p}\tau_k\to0,
\]
which implies that $\rho^{\Tk}$ also converges  to $\rho$ in $L^p(\Om \times (0,T))$.  
 Thus, up to a subsequence, 
\begin{equation}\label{pointwise_convergence}
\rho^{\Tk}, \tilde{\rho}^{\Tk}  \quad \textrm{converge pointwise to} \quad \rho \quad \textrm{for almost every $(x,t) \in \Om \times (0,T)$}.
\end{equation}

\noindent Using \eqref{H^1}, weak compactness, and \eqref{pointwise_convergence}  we have for all $\phi \in C^\infty(\OM)$
$$
\int_{t_1}^{t_2} \int_\Om 
\grad \phi \cdot \grad U'(\tilde{\rho}^{\Tk} (t))\tilde{\rho}^{\Tk}(t)  ~dxdt
=\int_{t_1}^{t_2} \int_\Om 
\grad \phi \cdot  \grad \Psi(\tilde{\rho}^{\Tk}(t)) ~dxdt
$$
\begin{equation}\label{RHS_seoncdterm_1}
\ra \int_{t_1}^{t_2} \int_\Om 
\grad \phi \cdot  \grad \Psi(\rho(t)) ~dxdt
= \int_{t_1}^{t_2} \int_\Om 
\grad \phi \cdot \grad U'(\rho(t))\rho(t)  ~dxdt.
\end{equation}

On the other hand, \eqref{strongLp}, \eqref{StrongLimit}, and \eqref{elliptic_estimate} imply that for a.e. $t\in(0,T)$,
$$
||c[\tilde{\rho}^{\Tk}(t)] - c[\rho(t)]||_{C^{1,\gam}(\Om)} \leq K_1  ||c[\tilde{\rho}^{\Tk}(t)] - c[\rho(t)]||_{W^{2,p}(\Om)}
\leq K_2 ||\tilde{\rho}^{\Tk}(t) - \rho(t)||_{L^p(\Om)} \ra 0.
$$
Hence, for a.e. $t\in(0,T)$, $\grad c[\tilde{\rho}^{\Tk}(t)]$ converges to $\grad c[\rho(t)]$ in the sup-norm, which shows that $c[\rho(t)]$ satisfies the homogeneous Neumann boundary condition, that is,
\begin{equation}\label{Neumann_holds}
\grad c[\rho(t)] \cdot \nu =0
\end{equation}
and
\begin{equation}\label{RHS_secondterm_2}
\int_{t_1}^{t_2} \int_\Om \tilde{\rho}^{\Tk}(t)
   \grad c[{\tilde{\rho}^{\Tk}}(t)] \cdot \grad \phi ~dxdt 
\ra
\int_{t_1}^{t_2} \int_\Om \rho(t) \grad c[\rho(t)] \cdot \grad \phi ~dxdt       \quad   \forall \phi \in C^\infty(\OM)
\end{equation}
by using the dominated convergence theorem.

As we have already shown \eqref{LHS_conv},
\eqref{RHS_seoncdterm_1}, and   \eqref{RHS_secondterm_2}, it remains to show the convergence of the first term in the right hand side of \eqref{approximate_weak_solution}.
From \eqref{alltime_boundedness}, \eqref{pointwise_convergence}, and dominated convergence theorem, we obtain
$$
\int_{t_1}^{t_2} \int_\Om \frac{\sqrt{\rho^{\Tk} (t)}(\sqrt{\rho^{\Tk}(t)} + \sqrt{\tilde{\rho}^{\Tk}(t) } )}{2} F'(\rho^{\Tk} (t)) \phi
\ra
\int_{t_1}^{t_2} \int_\Om \rho(t) F'(\rho(t)) \phi ~dxdt.
$$

Hence, for all $\phi \in C^\infty(\OM)$ and for all $0\leq t_1 \leq t_2 \leq T$, we have
$$
\int_\Om (\rho(t_2) - \rho(t_1) ) \phi ~dx      \qquad
$$
\begin{equation}\label{weak_solution}
= -\int_{t_1}^{t_2} \int_\Om \rho(t) \Big(  \grad U'(\rho(t) )  - \chi \grad c[\rho(t)]   \Big)\cdot \grad \phi ~dxdt -\int_{t_1}^{t_2} \int_\Om \rho(t) F'(\rho(t)) \phi ~dxdt.
\end{equation}
It is well known that \eqref{weak_solution} is equivalent to \eqref{real_weak_solution}.

Lastly, we prove \eqref{The_estimate}. Applying Fatou's lemma to \eqref{H^1}, we obtain
\begin{equation}\label{real_grad_estimate}
\int_0^T \int_{\Om} |\grad  \Psi(\rho)|^2
\leq \int_0^T \liminf_{k \ra \infty}  \int_{\Om} |\grad  \Psi(\tilde{\rho}^{\Tk})|^2
\leq
\liminf_{k \ra \infty}  \int_0^T  \int_{\Om} |\grad  \Psi(\tilde{\rho}^{\Tk})|^2
\leq  C,
\end{equation}
where we used \eqref{strongLp}, \eqref{StrongLimit}, and a lower semicontinuity of the functional $\rho \mapsto \int_\Om |\grad \Psi(\rho)|^2$ (see \eqref{weak_L2_LSC}).
Moreover, thanks to
\eqref{alltime_boundedness}, and \eqref{pointwise_convergence}, we get
\begin{equation}\label{real_alltime_boundedness}
\esssup_{\Om \times [0,T]} \rho(x,t) \leq C_1.
\end{equation}

\vspace{3mm}
\textit{Step 3. Removing the hypothesis \emph{\textbf{(H)}}} :
We show that the assumption \textbf{(H)} can be dropped.
Given $\rho_0 \in L^\infty(\Om)$ and a convex function $U$ satisfying Assumption \ref{assumptionsUF}, we approximate $\rho_0$ and $U$ with $(\rho_{0,\ep})_{\ep>0}$ and $(U_\del)_{\del>0}$ respectively. Let
$$
\rho_{0,\ep}:= (\tilde{\rho}_0 * \eta_\ep + \ep) \textbf{1}_{\OM}, \qquad U_\del(s) = U(s) + \del  s \log s,
$$
where  $\tilde{\rho}_0:=\rho_0$ in $\Om$ and it vanishes in $\R^d \backslash \Om$, and $\eta_\ep$ denotes a standard (positive) mollifier.
Thus, $(\rho_{0,\ep})$ and $(U_\del)$ satisfy \textbf{(H)}. Moreover, we note that
\begin{equation}\label{L_infty_norm_conv}
||\rho_{0,\ep}||_{L^\infty} \ra ||\rho_0||_{L^\infty} \quad \textrm{as} \quad \ep \ra 0.
\end{equation}

Therefore, we obtain a curve $\rho_\ed(t)$ which fulfills Theorem \ref{main} as shown in \textit{Step 1} and \textit{Step 2}.  
Hence, by \eqref{real_WFR_continuity}, \eqref{weak_solution}, \eqref{real_alltime_boundedness}, \eqref{real_grad_estimate}
$\rho_\ed$ satisfies

\begin{equation}\label{genuine_WFR_continuity}
\WFR(\rho_\ed(t), \rho_\ed(s)) \leq C_\ed |t-s|^{1/2},
\end{equation}
$$
\int_\Om (\rho_\ed(t_2) - \rho_\ed(t_1) ) \phi ~dx     + \int_{t_1}^{t_2} \int_\Om \grad \Psi_\del (\rho_\ed(t) )\cdot \grad \phi
 - \chi \rho_\ed(t) \grad c[\rho_\ed(t)] \cdot \grad \phi
$$
\begin{equation}\label{weak_solution_ed}
=  -\int_{t_1}^{t_2} \int_\Om \rho_\ed(t) F'(\rho_\ed(t)) \phi ~dxdt  \qquad \forall \phi \in C^\infty(\OM), \quad \forall 0\leq t_1 \leq t_2 \leq T,
\end{equation}

\begin{equation}\label{genuine_alltime_boundedness}
\esssup_{\Om \times [0,T]} \rho_\ed(x,t) \leq C_{1,\ep},
\end{equation}

\begin{equation}\label{genuine_grad_estimate}
\int_{\Om \times (0,T)} |\grad  \Psi_\del(\rho_\ed)|^2 =\int_{\Om \times (0,T)} |\rho_\ed \grad U_\del'(\rho_\ed)|^2  \leq   C_{\ep,\del},
\end{equation}
where we have used the notations $C_{1,\ep}$ and $C_\ed$, which depend on $\rho_{0,\ep}$ and $U_\del$, to distinguish them from $C_1$ and $C$.
Letting $\ed$ be sufficiently small, we can assume that $C_\ed \leq C$ for some constant $C>0$. In addition, $C_{1,\ep} \ra C_1$ as $\ep \downarrow 0$.

With an abuse of notation, we do not use sequence notations for $\ed$ like $\ep_k, \del_j$.

All we need to do is to repeat the proofs of \textit{Step 1} and \textit{Step 2} for $\rho_\ed$. The main differences are described below.

First, we note that
$\grad \Psi_\del (\rho_\ed) = \grad \Psi(\rho_\ed) + \del \grad \rho_\ed$ and $\Psi' \geq0$.
Then, from \eqref{genuine_grad_estimate}, we have
\begin{equation}\label{trick_estimate}
\int_{\Om \times (0,T)} |\grad  \Psi(\rho_\ed)|^2 
\leq
\int_{\Om \times (0,T)} |\grad  \Psi_\del(\rho_\ed)|^2 
 \leq   C.
\end{equation}

Once we have \eqref{genuine_WFR_continuity}, \eqref{genuine_alltime_boundedness}, and \eqref{trick_estimate}, the same argument as in \textit{Step 1} and \textit{Step 2} yields the following: $\rho_\ed(t)$ weakly* converges to a curve $\vrho(t): [0,T] \ra  \cM^+(\Om) \cap  L^\infty(\Om)$ satisfying 
\begin{equation*}
\WFR(\vrho(t), \vrho(s)) \leq C  |t-s|^{1/2}.
\end{equation*}
Moreover, the weak formulation in \eqref{weak_solution_ed} holds for $\vrho(t)$ instead of $\rho_\ed$, and
$$
\esssup_{\Om \times [0,T]} \vrho  (x,t) \leq C_1,
\qquad
\int_{\Om \times (0,T) } |\grad \Psi(\vrho)|^2 
=\int_{\Om \times (0,T) } |\vrho \grad U'(\vrho)|^2  
\leq C.
$$
This completes the proof.

\vspace{5mm}
\section{Appendix}\label{section_appendix}
\subsection{Proofs of Proposition \ref{CS_theorem}}
(i) Let $(\rho_n)_{n \in \N}$ be a minimizing sequence of the scheme \eqref{argmin}, and let $(c[\rho_n])_{n \in \N}$ be a sequence of the solutions in \eqref{elliptic} corresponding to $\rho_n$. Since $\rho_n \leq \overline{B}$, for any $p>1$ there exists
$\rho \in L^p(\Om)$ such that up to a subsequence
$\rho_n \wra \rho$ weakly in $L^p(\Om)$. 
Also, the functional $\rho \mapsto \int_\Om U(\rho) dx$ is lower semicontinuous with respect to such a weak convergence, since $\rho_n$ is uniformly bounded by $\overline{B}$
(see e.g. \cite[Proposition 7.7]{S15}).
On the other hand,  we know from \eqref{elliptic_estimate} that $||c[\rho_n]||_{W^{2,p}}$ is uniformly bounded in $n$. Using weak compactness in $W^{2,p}$ and the Rellich-Kondrachov theorem, up to a subsequence, we see that $\sup_\Om \big( |c[\rho_n] - c[\rho]| + |\grad c[\rho_n] - \grad c[\rho]| \big) \ra 0$ as $n \ra \infty$. Thus, $\int_\Om \Lam|c[\rho_n]|^2 + |\grad c[\rho_n]|^2 \ra \int_\Om \Lam|c[\rho]|^2 + |\grad c[\rho]|^2 $.
Since $\Om$ is bounded and $\rho_n$ narrowly converges to $\rho$, we have $W_2^2(g, \rho_n) \ra W_2^2(g, \rho)$.
Therefore, we conclude that
$$
\inf_{\mu \in \cS} \bigg(\cE_1(\mu) + \frac{W_2^2(g, \mu)}{\T} \bigg) = \liminf_{n \ra \infty} \bigg(\cE_1(\rho_n) + \frac{W_2^2(g, \rho_n)}{\T} \bigg) \geq \cE_1(\rho) + \frac{W_2^2(g, \rho)}{\T},
$$
which means that $\rho$ is a minimizer.\\

(ii)
Now we investigate an Euler-Lagrange equation for the minimization problem \eqref{argmin}:

Let $\rho$ be a minimizer obtained from \eqref{argmin}. 
Let $\tilde{\rho}$ be a nonnegative measure so that  $\tilde{\rho} \leq \overline{B}:=\frac{1}{\T \chi } $ and $|\tilde{\rho}|=|\rho|$, and  define $\rho_\ep:=(1-\ep) \rho + \ep \tilde{\rho}$ with $0<\ep <1/2$. From the minimality of $\rho$, we have
\begin{equation}\label{Appendix_first_variation1}
\cE_{1}(\rho) + \frac{W_2^2(\rho, g)}{\T} 
\leq 
\cE_{1}(\rho_\ep) + \frac{W_2^2(\rho_\ep, g)}{\T}.
\end{equation}
Recall that $\rho \leq \overline{B}$ and the function $sU'(s)$ is bounded from below on $(0,\infty)$ due to the assumptions on $U$. 
Following the arguments in \cite[Lemma 8.6]{S15}, we see that $U'(\rho) \in L^1(\Om)$. In addition, we note that 
$\big|
(\tilde{\rho}-\rho) U'((1-\ep)\rho + \ep \tilde{\rho})\big| \leq 2\overline{B} \big( |U'(\rho)| + |U'(\overline{B})| \big) \in L^1(\Om)$. 
Letting $\ep \downarrow 0$ and using the dominated convergence theorem, we get
\begin{equation}\label{Appendix_first_variation2}
\frac{\int U(\rho_\ep) - \int U(\rho)}{\ep} \ra
\int U'(\rho)(\tilde{\rho} - \rho).
\end{equation}
Moreover, from \eqref{elliptic}, we obtain
\begin{equation}\label{Appendix_first_variation3}
\frac{\int  \big( \Lam |c[\rho_\ep]|^2 + |\grad c[\rho_\ep]|^2 \big) 
- \int  \big( \Lam |c[\rho]|^2 + |\grad c[\rho]|^2 \big) }{\ep} 
\ra
\int 2 c[\rho ](\tilde{\rho} - \rho),
\end{equation}
where we used the fact that $c[\rho_\ep] = c[\rho] + \ep c[\tilde{\rho} -\rho]$ and $\int_\Om c[\tilde{\rho}] \rho = \int_\Om c[\rho] \tilde{\rho}$.

Let $\vphi \in \cK(\rho ,g)$ and $\vphi_\ep \in \cK(\rho_\ep ,g)$. We also note that
$$
W_2^2 (\rho_\ep,g)=\int_\Om \vphi_\ep d\rho_\ep + \int_\Om \vphi_\ep^c dg, \qquad
W_2^2(\rho,g) = \int_\Om \vphi d\rho + \int_\Om \vphi^c dg \geq \int_\Om \vphi_\ep d\rho + \int_\Om \vphi_\ep^c dg.
$$
Following the arguments in \cite[Lemma 3.4]{BS05}, we see that $\vphi_\ep$ converges uniformly, up to a subsequence, to $\vphi$ as $\ep \ra 0$ (see also \cite[Proposition 2.4]{KMW23}). Therefore, we have
\begin{equation}\label{Appendix_first_variation4}
\liminf_{\ep \ra 0}\frac{W_2^2(\rho_\ep, g) - W_2^2(\rho, g) }{\ep} 
\leq \liminf_{\ep \ra 0} \int_\Om \vphi_\ep d (\tilde{\rho}- \rho) = \int_\Om \vphi  d (\tilde{\rho}- \rho).
\end{equation}

Combining \eqref{Appendix_first_variation1}-\eqref{Appendix_first_variation4}, 
we obtain
$$
\int_\Om U'(\rho)\rho - \chi \int_\Om c [\rho] \rho + \frac{1}{\T}\int_\Om \varphi \rho 
\leq
\int_\Om U'(\rho) \tilde{\rho} - \chi \int_\Om c[\rho] \tilde{\rho} + \frac{1}{\T}\int_\Om \varphi \tilde{\rho}.
$$

Let $h:=U'(\rho) - \chi c[ \rho] + \frac{\varphi}{\T}$. Since
$\int h \rho \leq \int h \tilde{\rho}$ for all admissible $\tilde{\rho}$,
there exists $l \in \R$ such that

$$
 \begin{cases}
\rho=\overline{B} & \text{a.e. in $\{h<l \} $,}  \\     
\rho \leq \overline{B}      & \text{a.e. in $\{ h=l \}$,} \\
\rho=0 & \text{a.e. in $\{ h>l \}$.}
  \end{cases}
$$
More precisely, we can write 
$$
l:= \sup_{a \in \R} \{ \overline{B} \cL^d( \{ h<a \} )  \leq |\rho| \}.
$$

Similar to \cite[Lemma 8.6]{S15}, we see that supp$(\rho) = \Om$. Hence, $h \leq l$ a.e. So we have 
\begin{equation}\nonumber
U'(\rho)=
  \begin{cases}
U'(\overline{B}) & \text{a.e. in $\{ h<l \}$,}  \\           
   l+\chi c - \vphi/ \T       & \text{a.e. in $\{h=l\}$.  }
  \end{cases}
\end{equation}
Recalling that $U'$ is increasing and distinguishing two cases $\rho= \overline{B}$ and $\rho<\overline{B}$, we are able to redefine (up to a negligible set)
\begin{equation}\label{U'_regularity}
U'(\rho) := \min (U'(\overline{B}), l + \chi c - \frac{\vphi}{\T}) \quad \textrm{on} \quad \Om.
\end{equation}
We know that the Kantorovich potential $\vphi$ is Lipschitz on $\OM$. Moreover,  $c$ is also Lipschitz because $c \in W^{2,q}(\Om)$ for large $q$. Since 
$U'(\overline{B})$ and  $l + \chi c - \frac{\vphi}{\T}$ are Lipschitz  on $\OM$,  $U'(\rho)$ is also Lipschitz by \eqref{U'_regularity}. 
If we define $p:=(l-h)^+$ with $U'(\rho)$ redefined, then $p$ is continuous and for all $x \in \OM$,
\begin{equation}\label{Appendix_EL_conclusion}
    U'(\rho) - \chi c[\rho] + \frac{\vphi}{\T} +p=l, \quad p \geq0,  ~p(\overline{B} -\rho)=0.
\end{equation}

\noindent Since $U'(0+)=-\infty$, \eqref{Appendix_EL_conclusion} implies
$$
\essinf_\Om  \rho>0, 
$$
and hence $\rho$ is also Lipschitz. \\

(iii)
From the assumptions and (ii), we have $g,\rho \in C^{0,a}(\OM)$, and $g$ and $\rho$ are strictly positive and bounded. Then Caffarelli's regularity theory gives $\vphi \in C^{2, a}(\OM)$ (see also \cite[Theorem 4.14]{V03}).
Let $w:=\vphi - \T \chi c$ and $x_0 \in \argmin_{\OM} w$. Using  \eqref{Appendix_EL_conclusion}, we know that $\rho$ has its maximum at $x_0$, and hence $\rho(x_0)=||\rho||_{L^\infty}$.
We now consider two cases, $x_0 \in \textrm{int}(\Om)$ and $x_0 \in \p \Om$.

First, let $x_0$ belong to the interior of $\Om$. Since $w$ has its minimum at $x_0$, we have $\Del w(x_0) \geq0$, and hence $\Del \vphi(x_0) \geq \T \chi (c(x_0)- \rho(x_0)) \geq -\T \chi \rho(x_0)$. 
Due to the connection between the map $T$ and the Monge-Ampère equation, we have
$\rho(x_0) = g(T(x_0)) \textrm{det} (I_d - D^2 \vphi(x_0))$. Since $I_d -D^2 \vphi$ is positive definite, thanks to the arithmetic-geometric inequality, 
$\textrm{det}(I_d -D^2 \vphi(x_0)) \leq (1- \frac{\Del \vphi}{d})^d \leq (1+\frac{\T \chi \rho(x_0)}{d})^d$. Consequently, we obtain
$$
X \leq Y(1+X)^d, \quad \textrm{or} \quad Y\geq \frac{X}{(1+X)^d},
$$
where $X:=\frac{\T \chi ||\rho||_{L^\infty}}{d}$ and 
$Y:= \frac{\T \chi ||g||_{L^\infty}}{d}$.

Let $G: [0,\infty) \ra \R$ be the map defined as $G(x)=\frac{x}{(1+x)^d}$, and observe that  $G$ attains its maximum at $x=\frac{1}{d-1}$ (when $d\geq2$). Note that for given $\lam>1$, there exists a small $\del(\lam) \in (0,\frac{1}{2(d-1)})$ such that
if $x \in [0,\del(\lam)]$, then $G(x) \geq \frac{x}{1+ \lam d x}$.
We also note that $\rho \leq \overline{B}:=\frac{1}{\chi \T}$ implies $X \leq \frac{1}{d} <\frac{1}{d-1}$. Then, if $Y$ is small enough, the level set $\{ G \leq Y\}$ should be located in a small neighborhood of the origin. 
Therefore, we conclude that given $\lam>1$, if $Y < G(\del(\lam))$, i.e., $\T ||g||_{L^\infty} < c_0=c_0(\lam, \chi, d):=\frac{d \del(\lam)}{\chi (1+\del(\lam))^d}$,
then 
$$
Y \geq \frac{X}{(1+X)^d} \geq \frac{X}{1+ \lam d X},
$$
because $X \leq \del(\lam)$. Thus, we have $\frac{1}{X} + \lam d \geq \frac{1}{Y}$, which gives \eqref{CS_estimate}.

Second, we need to check the case $x_0 \in \p \Om$. Since $w$ attains its minimum at $x_0$, we deduce $\frac{\p w}{\p \nu}(x_0) \leq 0$. So we have $\frac{\p \vphi}{\p \nu}(x_0) \leq 0$ because of $\frac{\p c}{\p \nu}(x_0) = 0$. Hence, $\frac{\p \vphi}{\p \nu}(x_0)=(x_0 - T(x_0)) \cdot n \leq 0$, where $T$ is the optimal transport map from $\rho$ to $g$. On the other hand, since we assumed $\Om$ is strictly convex, we obtain $(T(x_0) - x_0) \cdot n <0$ unless $T(x_0)=x_0$. Consequently, if $x_0 \in \p \Om$, then we have a contradiction except for the case $T(x_0)=x_0$. 
Finally, we have to consider the case where $T(x_0)=x_0$. If so, we get $\frac{\p w}{\p \nu}(x_0) = 0$. Due to the minimality of $x_0$ for $w$, we know that $\Del w(x_0) \geq 0$. Then we can obtain \eqref{CS_estimate} by following the same proof for the case $x_0 \in \textrm{int}(\Om)$. 
\qed

\section*{Acknowledgments}
K. Kang was supported by NRF RS-2024-00336346 and RS-2024-00406821.
H. Kim was supported by NRF-2021R1F1A1048231.   
G. Seo was supported by the National Research Foundation of Korea (NRF) grant funded by the Korea government (MSIT) (RS-2023-00219980 and RS-2023-00212227).



\end{document}